\documentclass{amsproc}
\usepackage{amsmath,amsthm,amssymb, amscd}
\DeclareSymbolFont{bbold}{U}{bbold}{m}{n}
\DeclareSymbolFontAlphabet{\mathbbold}{bbold}
\newtheorem{theorem}{Theorem}[section]
\newtheorem{prop}[theorem]{Proposition}
\newtheorem{lemma}[theorem]{Lemma}
\newtheorem{cor}[theorem]{Corollary}
\newtheorem{conj}[theorem]{Conjecture}

\newtheorem{thm}[theorem]{Theorem}
\theoremstyle{definition}
\newtheorem{definition}[theorem]{Definition}

\theoremstyle{remark}
\newtheorem{remark}[theorem]{Remark}

\numberwithin{equation}{section}
\usepackage{graphicx}

\input xy
\xyoption{all}
\usepackage{pb-diagram}
\usepackage[all]{xy}
\input xy
\xyoption{all}

\begin{document}

\title{Homological aspects of branching laws}

\author{Dipendra Prasad}
\address{
Indian Institute of Technology Bombay \\
Powai \\
Mumbai - 400 076 \\ \
India}
\email{prasad.dipendra@gmail.com}

\thanks{}



\maketitle

\newcommand{\Sc}{\mathcal{S}}
\newcommand{\V}{\mathcal{V}}
\newcommand{\M}{\mathfrak{R}}
\newcommand{\W}{\mathcal{W}}
\newcommand{\Hecke}{\mathcal{H}}

\newcommand{\PP}{\mathcal{P}}
\newcommand{\Q}{\mathbb{Q}}
\newcommand{\R}{\mathbb{R}}
\newcommand{\Z}{\mathbb{Z}}

\newcommand{\Gm}{\mathbb{G}_m}

\newcommand{\CC}{\mathbb{C}}
\newcommand{\FF}{\mathbb{F}}

\newcommand{\F}{\mathfrak{F}}
\newcommand{\N}{\mathbb{N}}
\newcommand{\RR}{\mathcal{R}}

\newcommand{\Tor}{{\rm Tor}}
\newcommand{\Hom}{{\rm Hom}}
\newcommand{\EP}{{\rm EP}}
\newcommand{\Bes}{{\rm Bes}}
\newcommand{\PD}{{\rm PD}}
\newcommand{\Ps}{{\rm Ps}}
\newcommand{\Ext}{{\rm Ext}}
\newcommand{\Ind}{{\rm Ind}}
\newcommand{\ind}{{\rm ind}}
\def\circG{{\,^\circ G}}

\def\GG{{\rm GG}}
\def\G{{\rm G}}
\def\Aut{{\rm Aut}}
\def\SL{{\rm SL}}
\def\Spin{{\rm Spin}}
\def\PSp{{\rm PSp}}
\def\PSO{{\rm PSO}}
\def\PGSO{{\rm PGSO}}
\def\PSL{{\rm PSL}}
\def\GSp{{\rm GSp}}
\def\PGSp{{\rm PGSp}}
\def\Sp{{\rm Sp}}
\def\sc{{\rm sc}}
\def\St{{\rm St}}
\def\GU{{\rm GU}}
\def\SU{{\rm SU}}
\def\U{{\rm U}}
\def\Wh{{\rm Wh}}
\def\GO{{\rm GO}}
\def\GL{{\rm GL}}
\def\PGL{{\rm PGL}}
\def\GSO{{\rm GSO}}
\def\Gal{{\rm Gal}}
\def\SO{{\rm SO}}
\def\OO{{\rm O}}
\def\Out{{\rm Out}}
\def\O{{\mathcal O}}
\def\Sym{{\rm Sym}}
\def\tr{{\rm tr\,}}
\def\ad{{\rm ad\, }}
\def\Ad{{\rm Ad\, }}
\def\rank{{\rm rank\,}}

\begin{flushright}{\today}
\end{flushright}

\begin{abstract}
  In this mostly expository article,  we consider certain homological aspects
  of branching laws for representations of a group restricted to its subgroups in the context of $p$-adic groups.
  We follow our earlier paper \cite{Pr3} updating it with some more recent works. In particular,
  following Chan and Chan-Savin, see many of their papers listed in the bibliography,
  we have emphasized  in this work that the restriction of a (generic) representation
  $\pi$ of a group $G$ to a closed subgroup $H$ (most of the paper is written in the context of GGP)
  turns out to be a projective representation
  on most Bernstein blocks of the category of
  smooth representations of $H$. Further, once $\pi|_H$ is a projective module in a particular
  Bernstein block, it has a simple structure. 
  \end{abstract}

\maketitle

\setcounter{tocdepth}{1}

\tableofcontents

\newpage 
\section{Introduction}

    If $H$ is a subgroup of a group $G$, $\pi_1$ an irreducible representation
    of $G$, one is often interested in decomposing the representation $\pi_1$ when restricted to $H$ into irreducible components.
    This is what is called the 
  branching law of representations of $G$ restricted to $H$, much studied for
  finite dimensional representations of finite, compact or Algebraic groups.
  In this paper,  we will be dealing mostly with infinite dimensional representations of a  group $G$
  which when restricted to $H$ are  usually not completely reducible  and  there is often no obvious meaning to
  ``decomposing the representation restricted to $H$'', or a meaning has to be assigned in some precise way,
  such as the Plancherel decomposition of unitary representations of $G$ restricted to $H$.
Unless otherwise mentioned, we will say that a representation $\pi_2$ of $H$ appears in a representation $\pi_1$ of $G$ if
  \[\Hom_H[\pi_1,\pi_2] \not = 0.\]

    The local GGP conjectures (which are all theorems now!) are about such branching laws for certain pairs of classical groups
    $(G,H)$, which in this paper we will often take to be $(\GL_{n+1}(F), \GL_n(F))$, or $(\SO_{n+1}(F), \SO_n(F))$, where $F$
    is a local field which will be non-archimedean unless otherwise mentioned.

   For an irreducible admissible representation $\pi_1$ of $\SO_{n+1}(F)$, and $\pi_2$ of $\SO_n(F)$, the question of interest for GGP is in the  understanding of the Hom spaces,
    \begin{eqnarray*}
      \Hom_{\SO_n(F)}[\pi_1,\pi_2] & \cong &   \Hom_{\SO_n(F)}[\pi_1 \otimes \pi_2^\vee, \CC ] \\
    &\cong&  \Hom_{\SO_{n+1}(F) \times \SO_n(F)}[\Sc(X), \pi_1^\vee \otimes \pi_2],\end{eqnarray*}  
        where
        $X= \SO_n(F)\backslash [\SO_n(F) \times \SO_{n+1}(F)],$ and $\Sc(X)$ denotes the space of compactly supported
        smooth functions on $X$. This way of formulation of ``branching laws'' has the advantage
        that instead of working with the $G \times H$ homogeneous and spherical variety $X = \Delta(H) \backslash
        [G\times H]$ as in the present case, one can work with any spherical variety ${\mathbb X}$
        for a group ${\mathbb G}$ (on which the group ${\mathbb G}$ may not act transitively) 
        and consider the branching law associated to the
        spherical variety $\mathbb X$ as the study of,
        \[  \Hom_{\mathbb G}[\Sc({\mathbb X}), \Pi],\]
        for $\Pi$ an irreducible representation of the group ${\mathbb G}$.
        However, as the paper is mostly in the context of GGP, we will consider branching from  the point of view of
        $\Hom_{H}[\pi_1,\pi_2]  \cong    \Hom_{H}[\pi_1 \otimes \pi_2^\vee, \CC ].$

The first important result about branching laws considered by GGP  
is the multiplicity one property:
\[m(\pi_1,\pi_2):=\dim \Hom_{\SO_n(F)}[\pi_1,\pi_2] \leq 1.\]

This is due to  A. Aizenbud, D. Gourevitch, S. Rallis and G. Schiffmann in \cite{AGRS} in the non-archimedean case, 
and B. Sun and C. Zhu in \cite{Sun-Zhu} in the archimedean case.

It may be mentioned that before  
the full multiplicity one theorem was proved, even finite dimensionality  of the multiplicity spaces  was not 
known, which were later answered in greater generality in the work of  Y. Sakellaridis and   A. Venkatesh
in \cite{Sak-Ven}. For infinite dimensional representations which is what we are mostly dealing with,
there is also the possibility that  $m(\pi_1,\pi_2)$ could be identically 0 for a particular $\pi_1$!

With the multiplicity one theorems proved,
one then goes on to prove a more precise
description of the set of irreducible admissible representations $\pi_1$ of $\SO_{n+1}(F)$ and $\pi_2$ of $\SO_n(F)$
with  
\[\Hom_{\SO_n(F)}[\pi_1,\pi_2] \not = 0.\] 

Precise theorems about  $\Hom_{\SO_n(F)}[\pi_1,\pi_2]$ have become available in a series of papers due to 
Waldspurger and Moeglin-Waldspurger, cf. \cite{Wa}, \cite{Wa1}, \cite{Wa2}, \cite{Mo-Wa} for orthogonal groups. These  
were followed by a  series of papers by 
Beuzart-Plessis  for unitary groups, cf. \cite{Ra1},  \cite{Ra2},  \cite{Ra3}.

Given the interest in the space 
\[\Hom_{\SO_n(F)}[\pi_1,\pi_2] \cong  \Hom_{\SO_{n+1}(F) \times \SO_n(F)}[\Sc(X), \pi_1^\vee\otimes \pi_2],\] 
it is natural
to consider the related spaces 
\[\Ext^i_{\SO_n(F)}[\pi_1,\pi_2] \cong  \Ext^i_{\SO_{n+1}(F) \times \SO_n(F)}[\Sc(X), \pi_1^\vee \otimes \pi_2],\]  and in fact 
homological algebra methods suggest that the simplest answers are not for
these individual spaces, but for the 
alternating sum of their dimensions:  
$$\EP[\pi_1,\pi_2] = \sum_{i=0}^{\infty}(-1)^i\dim \Ext^i_{\SO_n(F)}[\pi_1,\pi_2];$$
these hopefully more manageable 
objects - certainly more flexible - when coupled with vanishing of higher $\Ext$'s (when available) 
may give theorems about 
$$\Hom_{\SO_n(F)}[\pi_1,\pi_2].$$

We hasten to add that before we can define $\EP[\pi_1,\pi_2]$, 
$\Ext^i_{\SO_n(F)}[\pi_1,\pi_2]$ 
need to  be proved to be finite dimensional 
for all $i \geq 0$ when $\pi_1$ and $\pi_2$ are finite length admissible representations 
of $\SO_{n+1}(F)$ and $\SO_n(F)$ respectively; further, $\Ext^i_{\SO_n(F)}[\pi_1,\pi_2]$ 
need to be proved to be 0 for $i$ large.

Vanishing of 
\[\Ext^i_{\SO_n(F)}[\pi_1,\pi_2]\]
for large $i$ is a well-known generality:  
for reductive $p$-adic groups $G$ considered here, it is 
known that 
\[\Ext^i_G[\pi,\pi'] = 0 \]
for any two smooth representations $\pi$ and $\pi'$ of $G$ when $i$ is greater than the $F$-split rank of $G$.  
This is a standard application of the projective resolution of the trivial representation
$\CC$ of $G$ provided by the (Bruhat-Tits) building associated to $G$.

For the proof  of the finite dimensionality 
of $\Ext^i_G[\pi_1,\pi_2]$
we  note 
that unlike the Hom spaces, $\Hom_{G}[\pi_1,\pi_2]$, where we will have no idea how to
prove finite dimensionality of $\Hom_{G}[\pi_1,\pi_2]$ if both  $\pi_1$ and $\pi_2$ are cuspidal,  
for $\Ext^i_G[\pi_1,\pi_2]$ exactly
this case can be handled a priori, for $i> 0$, as almost by the very  definition  of cuspidal
representations, they are  both projective and injective objects in the category of smooth representations
(and projective objects remain projective on restriction to a closed subgroup).

Finite dimensionality of $\Ext^i_{\SO(n)}[\pi_1,\pi_2]$ when one of the representations
$\pi_1,\pi_2$ is a full principal series representation, 
is achieved by an inductive argument
both on $n$ and on the split rank of the Levi from which the principal
series arises.   The resulting analysis needs  the notion of {\it Bessel models}, which is also
a restriction problem involving a subgroup which has both reductive and unipotent parts.

Recently, there is a very general 
finiteness theorem for $\Ext_G^i[\pi_1,\pi_2]$ (for spherical varieties) due to
A. Aizenbud and E. Sayag in \cite{AS}. However, the  approach via Bessel models which intervene when
analyzing principal series representations of $\SO_{n+1}(F)$ when
restricted to $\SO_n(F)$ has, as a  bonus, explicit answers about
Euler-Poincar\'e characteristics (at least in some cases).

For the definition and the theorem below, see Aizenbud and Sayag \cite{AS}.
\begin{definition}
  (Locally finitely generated representations) Suppose $G$ is a $p$-adic group and $\pi$ is a smooth representation of $G$. Then
  $\pi$ is said to be a {\rm locally finitely generated  representation} of $G$ (or, also, just locally finite representation) if it satisfies one of the following equivalent conditions.
  \begin{enumerate}
  \item For each compact open subgroup $K$ of $G$, $\pi^K$ is a finitely generated module of the Hecke-algebra ${\mathcal H}(K\backslash G /K)$.
  \item For each cuspidal datum $(M,\rho)$, i.e., $M$ a Levi subgroup of $G$, and $\rho$  a cuspidal representation of $M$, $\pi[M,\rho]$, the corresponding
    component of $\pi$ in the Bernstein decomposition of the category of smooth representations of $G$, is a finitely generated $G$-module. 
\end{enumerate}
  \end{definition}

Although the following theorem from \cite{AS} could be proved by other means for the restriction of a representation of $\GL_{n+1}(F)$ to $\GL_n(F)$ (using Bernstein-Zelevinsky filtration,
and locally finite generation of the Gelfand-Graev representation due to \cite{Bu-He}), the
generality in which it applies is noteworthy.

\begin{thm} \label{AS} (Aizenbud-Sayag)
  For $\pi$ an irreducible admissible representation of $\GL_{n+1}(F)$, the restriction
  of $\pi$ to $\GL_n(F)$ is locally finite. More generally, if $(G,H)$ with $H\subset H \times G$
  is a {\it spherical pair}, i.e., $H$ has on open orbit on the flag variety of $G\times H$, 
  where
  finite multiplicity is known such as for all the GGP pairs, then the restriction
  of an  irreducible admissible representation of $G$ to $H$ is locally finite.
    \end{thm}

  As a consequence of this theorem due to Aizenbud and Sayag, note that the restriction of an  irreducible
  representation $\pi$ 
  of $\GL_{n+1}(F)$ to $\GL_n(F)$ 
  is finitely generated in any Bernstein component of $\GL_n(F)$, hence $\pi|_{\GL_n(F)}$
  has nonzero irreducible quotients by generalities (a statement, which we said earlier, we will not know how to prove
  for a general restriction problem).
  
  The following corollary is an easy consequence of standard homological algebra where we also use the fact that
  if a module is finitely generated over a noetherian ring $R$ (which need not be commutative but contains 1),
  then it has a resolution by finitely generated projective
  $R$-modules.
  
  \begin{cor}
For $\pi_1$ an irreducible representation of $G=\GL_{n+1}(F)$, and 
    $\pi_2$ of $H=\GL_n(F)$ (and true more generally for {\it spherical pairs} $H \subset H \times G$ where
finite multiplicity is known),

\[ \Ext^i_H [\pi_1,\pi_2]\]
are finite dimensional, and zero beyond the split rank of $H$.
  \end{cor}

  We end the introduction by suggesting that although in this work we discuss exclusively
 the restriction problems arising
    in the GGP context, the notion of a locally finitely generated representation,
    and its becoming a projective module on restriction to  suitably chosen subgroups -- which is one of the properties
 emphasized   in this work -- should work
  well in many other situations involving finite multiplicities, such as the Weil representation and its restriction to dual reductive pairs which we
  briefly mention now. A criterion on locally finite generation, and projectivity, would be very welcome
  in the geometric context, say when a ($p$-adic) group $G$ acts on a ($p$-adic) space $X$ with an equivariant sheaf
  $\psi$, where one would like to understand  these questions  for the action of $G$ on the Schwartz space $\Sc(X,\psi)$.

  In the context of the Howe correspondence for a dual reductive pair $(G_1,G_2)$ inside $\Sp_{2n}(F)$,
  with $G_1$ ``smaller than or equal to'' $G_2$, with $K_1$, $K_2$ compact open subgroups in $G_1$ and $G_2$, 
    it appears that for the Weil representation $\omega$ of the metaplectic group ${\rm Mp}_{2n}(F)$, $\omega^{K_1 \times K_2}$ is a finitely generated module over both ${\mathcal H}(K_1\backslash G_1/K_1)$ and ${\mathcal H}(K_2\backslash G_2/K_2)$
    and is a projective module over ${\mathcal H}(K_1\backslash G_1/K_1)$. Moreover, it seems that
    one can use  $\omega^{K_1 \times K_2}$  as a bimodule
for ${\mathcal H}(K_1\backslash G_1/K_1)$ and ${\mathcal H}(K_2\backslash G_2/K_2)$
    to construct
    an embedding of the category of smooth representations of
${\mathcal H}(K_1\backslash G_1/K_1)$ 
     to the category of smooth representations of ${\mathcal H}(K_2\backslash G_2/K_2)$.
     Investigations on this ``functorial approach'' to the Howe correspondence seems not to have been undertaken so far. 
   
  \vspace{5mm}

\section{Branching laws from $\GL_{n+1}(F)$ to $\GL_n(F)$}

Recall
the following basic result  which is proved as a consequence of the Rankin-Selberg theory, cf. \cite{Pr2}.

\begin{thm} \label{duke93}Given 
an irreducible generic representation $\pi_1$ of $\GL_{n+1}(F)$, and an
irreducible generic representation $\pi_2$ of $\GL_{n}(F)$, 
\[\Hom_{\GL_n(F)}[\pi_1,\pi_2] \cong  \CC.\]
\end{thm}

The following theorem can be considered as the Euler-Poincar\'e version of the above theorem and is much more flexibile than the previous theorem, and proved more easily!  

\begin{thm}\label{whittaker} Let
  $\pi_1$ be an admissible representation of  $\GL_{n+1}(F)$ of finite length, and 
$\pi_2$ an admissible representation of  $\GL_{n}(F)$ of finite length.   
Then, $\Ext^i_{\GL_n(F)}[\pi_1,\pi_2]$ are finite dimensional vector spaces over $\CC$, and 
$$\EP_{\GL_n(F)}[\pi_1,\pi_2] 
= \dim {\rm Wh}(\pi_1) \cdot \dim {\rm Wh}(\pi_2),$$
  where ${\rm Wh}(\pi_1)$, resp. ${\rm Wh}(\pi_2)$,  denotes the space of Whittaker models for $\pi_1$, resp. $\pi_2$, 
with respect to fixed non-degenerate characters on a maximal unipotent subgroup
in $\GL_{n+1}(F)$ and $\GL_n(F)$ respectively. 

\end{thm}

Here is a curious corollary!

\begin{cor}
  { If
$\pi_1$ is an irreducible admissible representation of  $\GL_{n+1}(F)$, and 
$\pi_2$ an irreducible admissible representation of  $\GL_{n}(F)$, 
        then the only values taken by $\EP_{\GL_n(F)}[\pi_1,\pi_2]$ is 0 and  1,
    in particular it is $\geq 0$.}
  \end{cor}

{\bf Proof of Theorem  \ref{whittaker}}:  The proof of the Theorem \ref{whittaker} is  accomplished using some results of Bernstein and Zelevinsky, cf.
 \S3.5 of  \cite{BZ1},  regarding the
structure of representations of $\GL_{n+1}(F)$ restricted to the mirabolic subgroup.  

Recall that  $E_{n}$,  the mirabolic subgroup of $\GL_{n+1}(F)$, 
consists of matrices
in $\GL_{n+1}(F)$
whose last row is equal to $(0, 0,\cdots, 0, 1)$.  

For a representation $\pi$ of $\GL_{n+1}(F)$, Bernstein-Zelevinsky define
\[  \pi^i = \text{ the {\it i}-th derivative of $\pi$}, \]
which is a representation of $\GL_{n+1- i}( F)$.  
Of crucial importance is the fact that if $\pi$ is of finite length for $\GL_{n+1}(F)$,
then $\pi^i$ are representations of finite length of $\GL_{n+1-i}(F)$.

Bernstein-Zelevinsky prove that the restriction of an admissible representation $\pi$ of $\GL_{n+1}(F)$ to the mirabolic $E_n$ has a finite filtration whose successive quotients are described by the derivatives $\pi^i$ of $\pi$.  

Using the Bernstein-Zelevinsky filtration, and a form of Frobenius reciprocity
for Ext groups, Theorem \ref{whittaker} eventually follows from the following easy lemma. We refer to \cite{Pr3} for more details.  

\begin{lemma} \label{vanishing} If $V$ and $W$ are any two finite length representations of $\GL_d(F)$, then
  if $d>0$, $$\EP[V,W] = 0.$$  
  If $d=0$, then of course
\[\EP[V,W] = \dim V \cdot \dim W.\] 
\end{lemma}

The following result conjectured by the author some years ago, cf. \cite{Pr3},
and recently proved by Chan and Savin in \cite{CS2},
  is at the root of why the simple and general result in Theorem \ref{whittaker} above  
  translates into   a simple result about Hom spaces for generic representations
in Theorem \ref{duke93}.  

\begin{thm} \label{vanishing} Let $\pi_1$ be an irreducible generic representation of $\GL_{n+1}(F)$, and $\pi_2$ an
irreducible generic representation of $\GL_{n}(F)$. Then,
\[\Ext^i_{\GL_n(F)}[\pi_1,\pi_2] = 0, \]
for all $i > 0$. 
\end{thm}

On the other hand, Theorem \ref{whittaker} also has implications for the non-vanishing of (higher) Ext groups in certain cases that we discuss in the following remark.

\begin{remark} \label{rem1}
  One knows, cf. \cite{Pr2},  that there are irreducible generic representations  of $\GL_{3}(F)$ which have the trivial
  representation of $\GL_2(F)$ as a quotient;  similarly, there are
  irreducible non-generic representations of $\GL_{3}(F)$   with irreducible generic
  representations of $\GL_2(F)$ as a quotient.   For such pairs $(\pi_1,\pi_2)$ of representations, it
  follows from Theorem \ref{whittaker} on Euler-Poincar\'e characteristic that   
 \[\EP_{\GL_2(F)}[\pi_1, \pi_2]=0,\]  
 whereas
 \[\Hom_{\GL_2(F)}[\pi_1, \pi_2] \not = 0.\]  
 Therefore, for such pairs $(\pi_1,\pi_2)$
  of irreducible representations,  we must have 
  \[\Ext^i_{\GL_2(F)}[\pi_1, \pi_2] \not =0,\]
  for some $i>0$. The paper \cite{GGP2} studies more generally branching problem $\Hom_{\GL_n(F)}[\pi_1,\pi_2]$
  when one of the irreducible representations,  $\pi_1$ of $\GL_{n+1}(F)$ or $\pi_2$ of $\GL_n(F)$,  is not generic,  and both are parabolically induced from unitary Speh modules of
  Levi subgroups, i.e., are representations of the Arthur type, thus leading to
  non-vanishing of higher Ext groups.
\end{remark}

\section{An erratum to Kunneth Theorem from \cite{Pr3}}

It has been pointed out to me by G. Savin that in  the generality in which the Kunneth theorem is
asserted in \cite{Pr3}, it is not true. As the Kunneth theorem  plays an important role in many inductive arguments,
I take the occasion to state and prove a  corrected version here. The proof is only a minor
addition to what was given in \cite{Pr3}, but I  give a few more details here than there.

During the course of the proof of the Kunneth theorem, we will need to use the following 
most primitive form of  Frobenius reciprocity as was used in \cite{Pr3}.

\begin{lemma} \label{Frob}Let $K$ be an open subgroup of a $p$-adic group $G$. Let $E$ be a smooth representation of 
$K$, and $F$ a smooth representation of $G$. Then,
$$\Hom_G[ {\rm ind}_{K}^GE, F] \cong \Hom_{K} [E, F].$$
\end{lemma}

\begin{cor}
  For $K$ any compact open subgroup of a $p$-adic group $G$, and $E$ any smooth representation of $K$,
  ${\rm ind}_{K}^G E$ is a projective module for $G$.
  \end{cor}

We will need to use the following lemma in the next theorem.

\begin{lemma} \label{projective-product}
Let 
 $G_1$ and $G_2$ be two $p$-adic groups. Let $P_1, P_2$ be any two 
smooth projective modules for $G_1$, and $G_2$ respectively. Then  $P_1 \boxtimes P_2$ is a projective module for $G_1 \times G_2$.
\end{lemma}

\begin{proof}
  A slight subtlety in the proof of this lemma is that if $\tau$ is a smooth representation of $G_1 \times G_2$, though we have:
  \[ \Hom_{G_1\times G_2}(P_1 \boxtimes P_2, \tau) =  \Hom_{G_1}(P_1, \Hom_{G_2}( P_2, \tau))  =  \Hom_{G_1}(P_1, \Hom_{G_2}( P_2, \tau)^\infty),  \]
  which would seem to prove the lemma except that $\Hom_{G_2}( P_2, \tau)$ is not necessarily a smooth module for $G_1$, and
  taking smooth vectors $\pi \rightarrow \pi^\infty$ is not an exact functor on the category of $G_1$-modules.

  However, as we show, if $P_2$ is a smooth projective module for $G_2$,
  \[\tau \rightarrow \Hom_{G_2}( P_2, \tau)^\infty,\]
  is an exact functor from the category of smooth $(G_1\times G_2)$-modules to the category of smooth $G_1$-modules, i.e.,
  if $\tau_1\rightarrow \tau_2$ is a
  surjective map  of $(G_1\times G_2)$-modules,  so is  the induced map $\Hom_{G_2}( P_2, \tau_1)^\infty \rightarrow  \Hom_{G_2}( P_2, \tau_2)^\infty$,
  on $G_1$-smooth vectors in the corresponding Hom spaces.

  We prove that if  $\phi_2 \in \Hom_{G_2}( P_2, \tau_2)^\infty$ is $K_1$-invariant for $K_1$ a compact open subgroup in $G_1$,
  then there exists  $\phi_1 \in \Hom_{G_2}( P_2, \tau_1)^\infty$ which  is $K_1$-invariant whose image under the given map
  $\tau_1\rightarrow \tau_2$ is $\phi_2$.
  
  Note that as  $\tau_1\rightarrow \tau_2$ is a surjective map of smooth $(G_1\times G_2)$-modules,
  so is the induced map  $\tau_1^{K_1}\rightarrow \tau_2^{K_1}$ of  $G_2$ modules, and that
  there is an identification of $ \Hom_{G_2}( P_2, \tau_2)^{K_1}$ with
  $ \Hom_{G_2}( P_2, \tau_2^{K_1})$.
Now, the commutative diagram
\[  \begin{CD}
 \Hom_{G_2}( P_2, \tau_1^{K_1})   @>\text{}>>  \Hom_{G_2}( P_2, \tau_1)   \\
@V\text{}VV  @VVV  \\
 \Hom_{G_2}( P_2, \tau_2^{K_1})  @>>\text{}>   \Hom_{G_2}( P_2, \tau_2),  \end{CD} \]
in which the left vertical arrow is surjective by the projectivity of $P_2$ as a $G_2$-module, completes the proof of the assertion that
 if $P_2$ is a smooth projective module for $G_2$,
  \[\tau \rightarrow \Hom_{G_2}( P_2, \tau)^\infty,\]
  is an exact functor from the category of smooth $(G_1\times G_2)$-modules to the category of smooth $G_1$-modules. Hence, by the identification,
  \[ \Hom_{G_1\times G_2}(P_1 \boxtimes P_2, \tau) =  \Hom_{G_1}(P_1, \Hom_{G_2}( P_2, \tau))  =  \Hom_{G_1}(P_1, \Hom_{G_2}( P_2, \tau)^\infty),  \]
$P_1 \boxtimes P_2$ is a projective module for $G_1\times G_2$.
\end{proof}

\begin{remark}
  Considerations similar to the proof of Lemma \ref{projective-product} were made in the proof of Lemma 5.14 in \cite{APS}, which
  proved that $\Ext^i_{G_2}(\tau,X)^{K_1} = \Ext^i_{G_2}(\tau^{K_1}, X)$ for $\tau$ a smooth representation of $G_1 \times G_2$, and $X$ of $G_2$,
  and where $K_1 \subset G_1$ is a compact open subgroup.
\end{remark}

\begin{thm} \label{kunneth} Let $G_1$ and $G_2$ be two $p$-adic groups. Let $E_1, F_1$ be any two 
smooth representations of $G_1$, and $E_2,F_2$ be any two smooth representations
of $G_2$. Assume  that one of the following two
conditions hold:

\begin{enumerate}
\item Both the representations $E_1$ and $F_1$ of $G_1$ have finite lengths, and  that $G_1$ is a reductive $p$-adic group. Or,
\item The representation $E_1$ of $G_1$ and $E_2$ of $G_2$ have finite lengths,
  and both the groups $G_1,G_2$ are  reductive $p$-adic groups.
\end{enumerate}
Then,
\[\Ext^i_{G_1 \times G_2}[E_1\boxtimes E_2, F_1\boxtimes F_2] \cong \bigoplus_{i=j+k} \Ext^j_{G_1}[E_1,F_1] \otimes \Ext^k_{G_2}[E_2,F_2].\]
\end{thm}

\begin{proof} If $P_1$ is a projective module for $G_1$, and $P_2$ a projective module for $G_2$, then
 by Lemma \ref{projective-product},   $P_1 \boxtimes P_2$ is a projective module for $G_1 \times G_2$.

Let
$$\cdots \rightarrow P_1 \rightarrow P_0 \rightarrow E_1 \rightarrow 0,$$ 
$$\cdots \rightarrow Q_1 \rightarrow Q_0 \rightarrow E_2 \rightarrow 0,$$
be a projective resolution for $E_1$ as a $G_1$-module, and  a projective resolution
for $E_2$ as a $G_2$-module.

It follows that the tensor product of these two exact sequences: 
$$\cdots \rightarrow P_1\boxtimes Q_0 + P_0 \boxtimes Q_1  \rightarrow P_0\boxtimes Q_0  \rightarrow E_1 \boxtimes E_2 
\rightarrow 0,$$ 
 is a projective resolution of $E_1 \boxtimes E_2$. Therefore, 
$\Ext^i_{G_1 \times G_2}[E_1\boxtimes E_2, F_1\boxtimes F_2]$ can be calculated by taking the cohomology of the 
chain complex  \[\Hom_{G_1 \times G_2}[\bigoplus_{i+j = k} P_i\boxtimes Q_j, F_1\boxtimes F_2].\]

At this point, we use the well-known result that as $G_1$ is a reductive $p$-adic group,
the category of smooth representations
of $G_1$ is noetherian, i.e., any submodule of a finitely generated module is finitely generated, cf. Theorem 4.19 of \cite{BZ1}.
(This is the only place in this proof which uses $G_1$ to be reductive, perhaps it is not necessary?)
As $E_1$ is of finite
length, hence finitely generated, there is a  surjection from a finitely generated projective module
$P_0 \rightarrow E_1 \rightarrow 0$ with  $P_0$ a sum of finitely generated modules of the form
$ {\rm ind}_{K_{0k}}^{G_1}(W_{0k})$.  The kernel of the map $P_0 \rightarrow E_1$ is a
finitely generated $G$-module, hence again there is a map $P_1\rightarrow P_0$
with  $P_1$ a sum of finitely generated modules of the form $ {\rm ind}_{K_{1k}}^{G_1}(W_{1k})$ whose image is the kernel of
the mapping $P_0 \rightarrow E_1$. Iterating this process, we get a projective  resolution
of $E_1$ by finitely generated projective modules $P_i=  \sum_k {\rm ind}_{K_{ik}}^{G_1}(W_{ik}) $,
and thus the projective resolution of $E_1$ is:
\[\cdots \rightarrow P_1= \sum_j {\rm ind}_{K_{1k}}^{G_1} (W_{1k})  \rightarrow P_0 = \sum_k {\rm ind}_{K_{0k}}^{G_1} (W_{0k}) \rightarrow E_1 \rightarrow 0.\]

Assume that we are in the case (1) of the theorem, thus $E_1,F_1$ are of finite lengths, in particular $F_1$ is admissible.

Now since $F_1$ is admissible, $\Hom_{K_{ik}} [W_{ik}, F_1]$ is finite dimensional, hence there is a $K_{ik}$-invariant
subspace $F_{1ik} \subset F_1$ with $\dim(F_{1ik})$ finite dimensional
such that  $\Hom_{K_{ik}} [W_{ik}, F_{1ik}] = \Hom_{K_{ik}} [W_{ik}, F_{1}]$.

Since $W_{ik}, F_{1ik}$ are finite dimensional, we have the
isomorphism
\begin{eqnarray*} \Hom_{K_{ik} \times G_2}[W_{ik} \boxtimes Q_j, F_1\boxtimes F_2] & \cong & 
  \Hom_{K_{ik} \times G_2}[W_{ik} \boxtimes Q_j, F_{1ik}\boxtimes F_2] \\
  & \cong &  \Hom_{K_{ik}} [W_{ik}, F_{1ik}] \otimes \Hom_{G_2}[Q_j, F_2], \\
  & \cong &  \Hom_{K_{ik}} [W_{ik}, F_1] \otimes \Hom_{G_2}[Q_j, F_2],
\end{eqnarray*}
where it is in the second isomorphism that we use  the finite dimensionality of  $W_i, F_{1ik}$ in the corresponding statement about vector spaces:
\[\Hom[W_{ik} \boxtimes Q_j, F_{1ik}\boxtimes F_2] \cong  \Hom [W_{ik}, F_{1ik}] \otimes \Hom[Q_j, F_2].\]

As $P_i = \sum_k {\rm ind}_K^{G_1}(W_{ik}),$ we have:
\begin{eqnarray*} 
\Hom_{G_1 \times G_2}[ P_i\boxtimes Q_j, F_1\boxtimes F_2] & = 
& \sum_k  \Hom_{G_1 \times G_2}[{\rm ind}_{K_{ik}}^{G_1}(W_{ik})
  \boxtimes Q_j, F_1\boxtimes F_2] \\
&\cong & \sum_k \Hom_{K_{ik} \times G_2}[W_{ik} \boxtimes Q_j, F_1\boxtimes F_2] \\
&\cong &  \sum_k \Hom_{K_{ik}} [W_{ik}, F_1] \otimes \Hom_{G_2}[Q_j, F_2] \\
&\cong &  \Hom_{G_1} [P_i, F_1] \otimes \Hom_{G_2}[Q_j, F_2]. 
\end{eqnarray*}

We have a similar isomorphism in case (2) of the theorem when 
the representation $E_1$ of $G_1$ and $E_2$ of $G_2$ have finite lengths
in which case we resolve both $E_1$ and $E_2$ by finitely generated
$G_1$ and $G_2$ modules just as we resolved $E_1$ for $G_1$ above.

Thus we are able to identify the chain complex  $\Hom_{G_1 \times G_2}[\bigoplus_{i+j = k} P_i\boxtimes Q_j, F_1\boxtimes F_2]$ as the tensor product of the chain complexes $\Hom_{G_1}[P_i,F_1]$ and $\Hom_{G_2}[Q_j,F_2]$.
Now the abstract Kunneth theorem which calculates the cohomology of the tensor product of two chain complexes
in terms of the cohomology of the individual chain complexes completes the proof of  the theorem. \end{proof}

\begin{cor} \label{kunneth-cor}
Let $G_1$ and $G_2$ be two $p$-adic groups. Let $E_1, F_1$ be any two 
smooth representations of $G_1$, and $E_2,F_2$ be any two smooth representations
of $G_2$. Assume  that one of the two 
conditions in Theorem \ref{kunneth} hold. Then if $\EP_{G_1}[E_1,F_1]$ and  $\EP_{G_2}[E_2,F_2]$ are defined
(i.e., the corresponding $\Ext$ groups are finite dimensional and zero for large enough degree), then
$\EP_{G_1 \times G_2}[E_1\boxtimes E_2, F_1\boxtimes F_2]$ is defined (i.e., the $\Ext$ groups are finite dimensional and
zero for large enough degree), and we have:
\[\EP_{G_1 \times G_2}[E_1\boxtimes E_2, F_1\boxtimes F_2] = \EP_{G_1}[E_1,F_1] \cdot \EP_{G_2}[E_2,F_2].\]
  \end{cor}
\section{Bessel subgroup} \label{bessel}
We will use Bessel subgroups, and Bessel models without defining them referring the reader to \cite{GGP},
except to recall that
    these are defined for the classical groups $\GL(V), \SO(V), \U(V)$, through a subspace $W\subset V$,
    with $V/W$ odd dimensional which in the case of $\SO(V)$ will be a 
    split quadratic space. In this paper we will use these subgroups only for $\SO(V)$; the case of other groups is very similar
    with mostly only notational changes.
    The Bessel subgroup $\Bes(V,W)$ (shortened to $\Bes(W)$ if $V$ is understood)
    is a subgroup of $\SO(V)$ of the form
    $\SO(W) \ltimes U= \SO(W) \cdot U $ where  $U$ is a unipotent subgroup of $\SO(V)$ normalised by $\SO(W)$ which comes 
    with a character  $\psi: U \rightarrow \CC^\times$ which is invariant under $\SO(W)$.
    The Bessel subgroup $\Bes(V,W) = \SO(W)$ if $ \dim (V/W)=1$. For a representation $\rho$ of $\SO(W)$,
    we denote by $\rho \otimes \psi$ the corresponding representation of $\Bes(W) = \SO(W) \cdot U $.
    Throughout the paper, we will use the notation $\Hom_{\Bes(V,W)}[\pi,\rho]$ for $\pi$ a finite length smooth representation
    of $\SO(V)$, and $\rho$ that of $\SO(W)$, for the following space of homomorphisms:
    \[ \Hom_{\Bes(V,W)}[\pi,\rho] = \{\ell: \pi \rightarrow \rho| \ell(su \cdot v) =
    \rho(s) \psi(u) \ell(v) \},\]
    for all $v\in \pi,s \in \SO(W),$ and  $u\in U$, where $\Bes(V,W) = \SO(W) \ltimes U $.

    The representation $\ind_{\Bes(V,
      W)}^{\SO(V)} (\rho \otimes \psi)$ of $\SO(V)$ 
will be called  a Gelfand-Graev-Bessel representation, and plays a prominent role in analysing the
restriction problem from $\SO(V^+)$ to $\SO(V)$ for $V^+$ a quadratic space containing  $V$ as a subspace
of codimension 1 such that $V^+/W$ is a split quadratic space of even dimension. By Frobenius reciprocity,
we have:

\[ \Hom_{\Bes(V,W)}[\pi, \rho^\vee] = \Hom_{\SO(V)}[ \ind_{\Bes(V,W)}^{\SO(V)} (\rho \otimes \psi), \pi^\vee]. \]

    \begin{prop} \label{proj} If $\rho$ is a finite length  representation of $\SO(W)$,
      then  the Gelfand-Graev-Bessel representation, \[ \ind_{\Bes(V,W)}^{\SO(V)} (\rho \otimes \psi),\]
      is a
      locally finitely generated representation of $\SO(V)$ which is 
      projective if, further, $\rho$ is cuspidal representation of $\SO(W)$ and $\SO(W)$ is not the split
      $\SO_2(F)$.
    \end{prop}

    \begin{proof}
      Projectivity of the Gelfand-Graev representation for any quasi-split group is due to Chan and Savin in the appendix to the
      paper \cite{CS3}. Let us remind ourselves a slightly delicate point. 
      By exactness of $U$-coinvariants, what is obvious is that $\Ind_{U}^{G} ( \psi)$ is an injective module for
      $U$ any unipotent subgroup of a reductive group $G$.
      That the dual of a projective module is an injective module is a generality, but this does not prove that
      $\ind_{U}^{G} ( \psi)$ is projective!

      Instead of directly proving that  $\ind_{U}^{G} ( \psi)$ is projective, Chan and Savin
      prove that $\Ext_G^i[\ind_{U}^{G} ( \psi), \sigma] = 0$ for all $\sigma$ and all $i > 0$. By generalities,  for  algebras ${\mathcal H}$ containing a finitely generated $\CC$-algebra  $Z$ in its center over which
      ${\mathcal H}$ is finitely generated as a   $Z$-module, for a finitely generated ${\mathcal H}$-module $M$,
      $\Ext_{\mathcal H}^i[M,N] = 0$ for $i>0$ and for all $N$,  if and only if this is true for
      finitely generated ${\mathcal H}$-module  $N$ and eventually  $\Ext_{\mathcal H}^i[M,N] = 0$ for $i>0$ and for all $N$,  if and only if $\Ext_{\mathcal H}^i[M,N] = 0$ for $i>0$ and for all $N$ of  finite length as an ${\mathcal H}$-module.
      (Clearly,  only irreducible $N$ are adequate!)

      Going from finitely generated to finite length is a generality that Chan and Savin discuss, and is also a consequence of
      Proposition 5.2 of
      \cite{NP} according to which 
  \[\Ext_{\mathcal H}^i[M,N]
  \otimes_Z \widehat{Z} \cong \Ext_{\mathcal H}^i[M,\widehat{N}]
  \cong \lim_{\leftarrow}\Ext_{\mathcal H}^i[M,N/{\mathfrak m} ^nN],\]
  where $\widehat{N} = \displaystyle{ \lim_{\leftarrow}}(N/{\mathfrak m}^nN)$. Therefore if 
$\Ext_{\mathcal H}^i[M,N] \not = 0$, $\Ext_{\mathcal H}^i[M,N/{\mathfrak m} ^nN] \not = 0$ for some $n\geq 1$.

  For all this, finite generation of
  $M$ is essential for which Chan and Savin quote the paper \cite{Bu-He} which proves that the Gelfand-Graev
  representations are locally finitely generated.

  In our case, we can appeal to Theorem \ref{AS} of Aizenbud-Sayag to prove  that
  the Gelfand-Graev-Bessel representation $ \ind_{\Bes(W)}^{\SO(V)} (\rho \otimes \psi)$ are locally
  finitely generated which we now elaborate upon; the rest of the argument of Chan-Savin in \cite{CS3} goes verbatim. Note of course that just as 
$\Ind_{U}^{G} ( \psi)$ is an injective module for
  $U$ any unipotent subgroup of a reductive group $G$,
  $ \Ind_{\Bes(W)}^{\SO(V)} (\rho \otimes \psi)$ is an injective module for $\SO(V)$
  if $\rho$ is cuspidal representation of $\SO(W)$ and $\SO(W)$ is not the split
      $\SO_2(F)$.

  Let $V^+= V + L$ where $L$ is a one dimensional quadratic space such that
  $V^+  = X + W + Y$ for $X,Y$ isotropic, perpendicular to $W$. Consider the representation  $\tau \times \rho$ of
  $\SO(V^+)$, a 
  parabolically induced representation of $\SO(V^+)$ from the parabolic with Levi subgroup $\GL(X) \times \SO(W)$ of the representation
  $\tau \boxtimes \rho$ where $\tau$ is any cuspidal representation of $\GL(X)$.  The standard Mackey theory applies to consider
  the restriction of the representation  $\tau \times \rho$ of
  $\SO(V^+)$ to $\SO(V)$ for which the open orbit contribution of $\tau \times \rho$ gives
  $ \ind_{\Bes(W)}^{\SO(V)} (\rho \otimes \psi)$, see Theorem 15.1 of \cite{GGP}. This realises $ \ind_{\Bes(W)}^{\SO(V)} (\rho \otimes \psi)$ as a
  submodule of  the representation  $\tau \times \rho$ of
$\SO(V^+)$ restricted to $\SO(V)$ which is locally finitely generated
by Theorem \ref{AS} of Aizenbud-Sayag. Since the rings which govern a Bernstein block are Noetherian rings, submodules of
  locally finitely generated representations are locally finitely generated, proving the proposition.   \end{proof}

    Note a particular case of this proposition.
    
    \begin{cor}\label{split-SO(W)}
      If $W\subset V$ is a codimension one subspace of $V$, a quadratic space, and $\rho$
      a finite length  representation of $\SO(W)$, then $\ind_{\SO(W)}^{\SO(V)} (\rho)$ is a locally finitely
      generated representation of $\SO(V)$
      which is projective if $\rho$ is cuspidal
      (and if $\dim(W)=2$, $W$ is not split). Also, similar assertions
      for $\GL_n(F), \U_n(F)$.
      \end{cor}

\section{What does the restriction really looks like!}

    So far, we have been discussing the question: which representations of $\GL_n(F)$ appear as a quotient of  an irreducible representation
    of $\GL_{n+1}(F)$. It is possible to have a more complete understanding of what a representation of $\GL_{n+1}(F)$ restricted to $\GL_n(F)$ looks like.

  Vanishing of Ext groups in many cases (but not in all cases!), suggests that
  the restriction to $\GL_n(F)$ of
  an irreducible admissible
  (generic) representation $\pi$ of $\GL_{n+1}(F)$   is close to being a  projective module without
  being one in all the cases.

  Since the category of smooth representations of $\GL_n(F)$ is  decomposed
  into blocks parameterized by the inertial equivalence classes of cuspidal datum $(M,\rho)$ in $\GL_n(F)$, one can
  ask if the projection of $\pi$ to the particular block, call it $\pi[M,\rho]$, is a projective module in that block.
  This appears 
  to be  an important question to understand: given an irreducible representation $\pi$ of $\GL_{n+1}(F)$, for which
  blocks $(M,\rho)$ in $\GL_n(F)$, is $\pi[M,\rho]$ a projective module?

    Before discussing the next proposition, recall that Bernstein-Zelevinski introduced the notion of derivatives of representations of the general linear group $\pi \rightarrow \pi^i$ taking a representation
    $\pi$ of  $\GL_{n+1}(F)$  to one of  $\GL_{n+1-i}(F)$ for all
$0 \leq i \leq (n+1)$, with $\pi^{n+1}$ nonzero if and only if $\pi$ has a Whittaker model. 
  The Bernstein-Zelevinski derivatives satisfy  the Leibnitz rule (in the Grothendieck group of representations
    of $\GL_{n+1}(F)$):
    \[ (\pi_1 \times \pi_2)^d = \sum_{i=0}^{d} \pi_1^{d-i} \times  \pi_2 ^i.\]
   Further, for an irreducible  cuspidal representation $\pi$ of $\GL_d(F)$, the only nonzero derivatives are $\pi^0=\pi$, and
    $\pi^d = \CC$. 

  Bernstein-Zelevinski introduced a 
  filtration on representations of $\GL_{n+1}(F)$ which describes
    the restriction of a representation $\pi$ of $\GL_{n+1}(F)$ to the mirabolic subgroup of $\GL_{n+1}(F)$
    in terms of the derivatives $\pi^i$ of $\pi$.
    The successive quotients of the filtration on $\pi$ (a representation of $\GL_{n+1}(F)$), as a representation of $\GL_n(F)$ are the parabolically induced representations
    $\nu ^{1/2} \pi^{i+1} \times \GG(n-i)$
    for $0 \leq i \leq n$,
    where $\GG(i)$ is the Gelfand-Graev representation of $\GL_i(F)$. 

    The proof of the following proposition is now clear from the  Bernstein-Zelevinski
  filtration on representations of $\GL_{n+1}(F)$ and the Leibnitz rule for derivatives:
    
    \begin{prop}\label{gln}
      Let $\pi$ be an irreducible generic representation of $\GL_{n+1}(F)$. Let $(M,\rho)$ be a cuspidal datum for $\GL_n(F)$, thus
      $M= \GL_{n_1}(F) \times \cdots \times \GL_{n_k}(F)$ with $n = n_1+\cdots + n_k$,
      is a Levi subgroup inside $\GL_n(F)$, and $\rho = \rho_1 \boxtimes \cdots \boxtimes \rho_k$ is a tensor product of irreducible
      cuspidal representations of $\GL_{n_i}(F)$. Assume that none of the cuspidal representations $\rho_i$ of $\GL_{n_i}(F)$
      appear in the cuspidal support of $\pi$ even after an unramified twist.
      Then $\pi|_{\GL_n(F)} [M,\rho]$
            is the
            $[M,\rho]$ component of the Gelfand-Graev representation $\GG(n)=\ind_N^{\GL_n(F)} \psi$;
            in particular, $\pi|_{\GL_n(F)} [M,\rho]$ is a projective representation of $\GL_n(F)$.
      \end{prop}

    Here is the corresponding result for classical groups, asserted for simplicity of notation
    only for $\SO(W) \subset \SO(V)$ where
    $W \subset V$ is a codimension 1 nondegenerate subspace of a quadratic space $V$ with  $\dim(V)=n+1$. This
    result like Proposition \ref{gln} is also a
    consequence of a Bernstein-Zelevinski like filtration (due to Moeglin and Waldspurger in \cite{Mo-Wa})
    on the restriction of a representation of $\SO(V)$ to $\SO(W)$ when the representation of $\SO(V)$ is
    induced from a maximal parabolic
    with Levi of the form  $\GL_m(F) \times \SO(W')$ of a representation of the form $\mu_1 \boxtimes \mu_2$, and using the
    Bernstein-Zelevinski filtration for $\mu_1$ restricted to a mirabolic in $\GL_m(F)$. The proposition below uses
the representation $\ind_{\Bes(W_0)}^{\SO(W)} (\rho_0 \otimes \psi)$, for
    $\rho_0$ a cuspidal representation of $\SO(W_0)$,  which    we called  a Gelfand-Graev-Bessel representation
      in section \ref{bessel},  and which  is a projective representation by Proposition \ref{proj}.
Here, $\Bes(W_0)$ is the Bessel subgroup inside $\SO(W)$, introduced in section \ref{bessel},
    where $W_0\subset W$ is a
    nondegenerate subspace of a quadratic space $W$ with  $W_0^\perp$ an odd dimensional hyperbolic space,

        \begin{prop} \label{son}
      Let $\pi$ be an irreducible admissible  representation of $\SO(V)$ which is the full induction of a cuspidal representation
      of a Levi subgroup of $\SO(V)$. Let $(M,\rho)$ be a cuspidal datum for $\SO(W)$, thus, 
      $M= \GL_{n_1}(F) \times \cdots \times \GL_{n_k}(F) \times \SO(W_0)$ with $n = 2n_1+\cdots + 2n_k+ \dim(W_0)$,
      is a Levi subgroup inside $\SO(W)$, and $\rho = \rho_1 \boxtimes \cdots \boxtimes \rho_k \boxtimes \rho_0$
      is a tensor product of irreducible cuspidal
      representations $\rho_i$ of $\GL_{n_i}(F)$ for $i \geq 1$, and $\rho_0$ is an irreducible cuspidal representation of $\SO(W_0)$.
      Assume that none of the cuspidal representations $\rho_i$ of $\GL_{n_i}(F)$
      appear in the cuspidal support of $\pi$ even after an unramified twist (no condition on $\rho_0$).
      Then $\pi|_{\SO(W)} [M,\rho]$       is the
      $[M,\rho]$ component of the Gelfand-Graev-Bessel representation $\ind_{\Bes(W_0)}^{\SO(W)} (\rho_0 \otimes \psi)$; in particular,  $\pi|_{\SO(W)} [M,\rho]$  is a projective representation
      of $\SO(W)$ if $\SO(W_0)$ is not the split $\SO_2(F)$.
      \end{prop}

    \begin{remark}
      We assumed $\pi$ in Proposition \ref{gln} to be generic as otherwise the assertion in the Proposition \ref{gln} will become
      empty, i.e.,  $\pi|_{\GL_n(F)} [M,\rho]$ will be zero if $\pi$ is nongeneric. However, in Proposition \ref{son}
      we do not assume that $\pi$ is generic.
                  \end{remark}

  The following theorem is due to Chan and Savin,  cf. \cite{CS1}, \cite{CS2}, 
  \cite{CS3}. 
  \begin{thm} \label{CS}

    \begin{enumerate}
      
\item     The restriction of an irreducible admissible representation
    $\pi$ of $\GL_{n+1}(F)$ to $\GL_n(F)$ is a projective module in a
    particular Bernstein block of smooth representations of $\GL_n(F)$
    if and only if $\pi$ itself is generic and
    all irreducible $\GL_n(F)$-quotients of $\pi$,
in that particular Bernstein block of smooth representations of $\GL_n(F)$,
    are generic.

\item     If $\pi_1,\pi_2$ are any two irreducible representations
    of $\GL_{n+1}(F)$ whose restrictions to  $\GL_n(F)$ are projective in a
    particular Bernstein block of smooth representations of $\GL_n(F)$, then
    $\pi_1$ and $\pi_2$ are isomorphic in that particular Bernstein block of smooth representations of $\GL_n(F)$.

  \item Let $\pi$ be a smooth projective module in the Iwahori block of $\GL_n(F)$ such that
    $\dim \Hom_{\GL_n(F)}[\pi,\pi']\leq 1$ for any irreducible representation $\pi'$ of $\GL_n(F)$.
    Then $\pi$ is isomorphic to either $\ind_{\GL_n(O_F)}^{\GL_n(F)}(\St)$ or $\ind_{\GL_n(O_F)}^{\GL_n(F)}(\CC)$, 
  where $\St$ denotes the Steinberg representation of $\GL_n(\FF_q)$ where $\FF_q$ is the residue
  field of $O_F$.

    \end{enumerate}
    
    \end{thm}

  \begin{remark}
    After this theorem of Chan and Savin, the unfinished tasks are:
    \begin{enumerate}
    \item Given an irreducible generic representation $\pi$ of $\GL_{n+1}(F)$, can we classify exactly the
      Bernstein blocks of $\GL_n(F)$ in which $\pi|_{\GL_n(F)}$ is not projective?
    \item  More generally, if $\pi$  is an irreducible representation of $\GL_{n+1}(F)$ which may or may not be generic,
      can one understand projective dimension (i.e., the minimal length of a projective resolution) of
      $\pi|_{\GL_n(F)}$ in a particular Bernstein block?
            \end{enumerate}

    As is often the case in representation theory  of $p$-adic groups,
    dealing with discrete series which are non-cuspidal is often the most difficult part.  In Proposition
    \ref{restriction} in the next section, 
    we prove that for $\pi$ a generic  representation of $\GL_{n+1}(F)$,  $\pi|_{\GL_n(F)}$ is a  projective
    representation in  those Bernstein blocks of $\GL_n(F)$ which contain no   non-cuspidal discrete series representations.
    Both  Proposition
    \ref{restriction} and  Proposition
    \ref{gln} can be considered as the simplest blocks where there is a nice answer.
  \end{remark}

 The following theorem of Chan, cf. \cite{Chan}, gives a complete classification of the irreducible
  representations of $\GL_{n+1}(F)$ which when restricted to $\GL_n(F)$ are projective modules, thus remain projective
  in {\it all} blocks.

\begin{thm} \label{thmchan}
    Let $\pi$ be an irreducible representation of $\GL_{n+1}(F)$. Then $\pi$ restricted to $\GL_n(F)$ is a projective representation
    if and only if
    \begin{enumerate}
    \item Either $\pi$ is essentially square integrable, or,
    \item $(n+1) =2d$, $\pi = \pi_1 \times \pi_2$ where $\pi_i$ are cuspidal on $\GL_{d}(F)$.
    \end{enumerate}
    \end{thm}

\begin{remark} \label{rem3} Observe that if an irreducible representation $\pi$ of $\GL_{n+1}(F)$ is non-generic or
  has a non-generic
  quotient representation $\pi'$ of $\GL_n(F)$, then as $\EP[\pi,\pi']=0$ by Theorem \ref{whittaker},
$\Ext^i_{\GL_n(F)}[\pi,\pi'] \not =0$ for some $i>0$, and
  $\pi$ cannot
  be a projective module when restricted to  $\GL_{n}(F)$. Now,
            the non-tempered GGP, conjectured in  \cite{GGP2} and proved for $\GL_n(F)$ in \cite{Chan2}, \cite{Gur},  describes irreducible representations  $\pi$ of $\GL_{n+1}(F)$ and  $\pi'$ of $\GL_n(F)$ (which are Speh representations on discrete series representations, i.e., have $A$-parameters) with 
      \[\Hom_{\GL_n(F)}[\pi,\pi'] \not = 0.\]

      This takes care of eliminating most representations ($\pi$ of $\GL_{n+1}(F)$ either non-generic or
    generic but  of Arthur type)
      not in the  list in Theorem \ref{thmchan} as non-projective except 
       for the tempered representation $\pi= \St_d
      \times \chi \St_d$
      of $\GL_{2d}(F)$ where $\chi$ is a unitary character of $F^\times$, and $\St_d$ is the Steinberg representation
      of $\GL_d(F)$ which by the
non-tempered GGP
      has no non-generic quotient of $\GL_{2d-1}(F)$ with an $A$-parameter (but it does have other
      non-generic quotients of $\GL_{2d-1}(F)$, cf. Theorem \ref{CS}(1)). On the other hand,   \[\Ext^i_{\GL_{2d-1}(F)}[\St_d \times \chi \St_d ,\mathbbold{1}_{d+1} \times \pi'] \not = 0,\]
            for $i=d,d-1$, where $\pi'$ is any tempered representation of $\GL_{d-2}(F)$,
see Theorem 7.4 of  \cite{Chan2}, and a recent work of Saad, cf. \cite{Sa}. Thus, one can re-interpret Theorem \ref{thmchan} to say that  an irreducible representation $\pi$ of $\GL_{n+1}(F)$ is not projective
      when restricted to $\GL_n(F)$ if and only if for some $i>0$,  \[\Ext^i_{\GL_n(F)}[\pi ,\pi'] \not = 0,\]
      for some irreducible non-generic representation $\pi'$ of $\GL_n(F)$ with $A$-parameter (a statement which is not yet quite proven
      as we do not know if any irreducible representation of $\GL_{n+1}(F)$ of dimension $>1$ has a quotient representation
      of $\GL_n(F)$ of Arthur type which we expect is the case).
    \end{remark}

  \begin{remark}
    From Theorem \ref{thmchan}  of Chan, it follows that just like cuspidal representations, discrete series representations of $\GL_{n+1}(F)$
    are always  projective representations when restricted to $\GL_n(F)$. This seems like a general feature of  all
    the GGP pairs. Thus for a GGP pair $(G,H)$, if $\pi$ is a discrete series representation
    of $G$, one expects $\pi|_H$ to be a projective representation of $H$; there
    is no proof yet of this expectation.
    \end{remark}
\section{A theorem of Roche and some consequences}

In the last section we discussed some situations where restriction of irreducible admissible
representations of $\GL_{n+1}(F)$ (resp., other classical groups) to $\GL_n(F)$ (resp., subgroups of other
classical groups) give rise to projective modules and which are for $\GL_{n+1}(F)$, by theorems of Chan and Savin,
very explicit compactly induced representations. In this section, we use a theorem of Alan Roche to one more such situation for both
$\GL_{n+1}(F)$ as well as for classical groups  where the restriction gives rise to projective modules. In this case,
however, the projective modules are {\it universal principal series} representations.

\begin{thm} \label{Roche} (Alan Roche) Let $G$ be a reductive $p$-adic group, $(M,\rho)$, a cuspidal datum.
  Assume
  that no nontrivial element of $N_G(M)/M$ preserves $\rho$ up to an unramified twist. Then
  the induced representation,
  \[ \Ind_P^G(\rho),\]
  is irreducible. Furthermore, the parabolic induction from $P$ (with Levi $M$) to $G$ gives an equivalence of categories
  \[ {        \RR}
    [M]{[\rho]} \rightarrow {\RR}[G]{[M, \rho]}.\]
  In particular, since the category of representations  ${\RR}[M]{[\rho]}$ in the Bernstein component of $M$ corresponding to the cuspidal representation $\rho$ of $M$ is the same as the category of modules over an Azumaya algebra with center 
  the ring of functions on the complex torus consisting of the unramified twists of $\rho$, the same is true
  of  the Bernstein component ${\RR}[G]{[M, \rho]}$ of $G$. 
\end{thm}

\begin{remark} \label{6.2}
  By the Geometric Lemma (which calculates Jacquet modules of full principal series representations, in this particular case when $\rho$ is cuspidal), the assertion
  that no nontrivial element of $N_G(M)/M$ preserves $\rho$ up to an unramified twist is equivalent to say that
  the Jacquet module with respect to the parabolic $P$ of the principal series representation
  $\Ind_P^G(\rho)$ contains $\rho$ with multiplicity 1, and
  no unramified twist of it distinct from itself.
  \end{remark}

\begin{prop} \label{restriction} Let $\G_{n+1}$ be any  of the classical groups $\GL_{n+1}(F),$  $\SO_{n+1}(F),$
  $\U_{n+1}(F)$.    Let $\pi_1$ be an irreducible representation of $\G_{n+1}(F)$ belonging to a generic $L$-packet of $\G_{n+1}(F)$,
  and let $(M,\rho)$  be a cuspidal
  datum for $\G_n(F)$. Assume that  no nontrivial element of $N_{\G_n(F)}(M)/M$ preserves $\rho$ up to an unramified twist.
  Let $\rho^0$ be an irreducible representation of $M^0$, the subgroup of $M$ generated by compact
  elements of $M$, with $\rho^0 \subset \rho|_{M^0}$.
  Then,
  the $(M,\rho)$ Bernstein component of $\pi_1$ restricted to $\G_n(F)$ is the
  ``universal principal series''
  representation, i.e., 
  \[ \pi_1|_{\G_n(F)} [M,\rho] \cong \ind_{P^0}^{\G_n}(\rho^0) \cong \Ind_{P}^{\G_n}\ind_{M^0}^M(\rho^0) ,\]
  where $P^0=M^0N$. In particular,  the $(M,\rho)$ Bernstein component of $\pi_1|_{\G_n(F)}$ is a projective representation which
  is independent of $\pi_1$.
\end{prop}

    \begin{proof}
      Since $M$ is a Levi subgroup of $\G_n(F)$, $M$ is a product of the groups $\GL_{n_i}(F)$ with 
      $\G_m(F)$ for some $m \geq 0$ (which are semisimple unless $\G_m$ is one of $\U_1(F)$, or $\SO_2(F)$), 
      it is easy to see that any irreducible representation of $M$ when restricted to $M^0$,
      is a finite direct sum of irreducible representations of $M^0$ with multiplicity 1. By the second adjointness combined with the Frobenius reciprocity for open subgroups, cf. Lemma \ref{Frob},
      for $\pi$ any irreducible representation of $\G_n(F)$,
      \[ \Hom_{\G_n(F)}[ \Ind_{P}^{\G_n(F)} \ind_{M^0}^M(\rho^0), \pi] \cong \Hom_M [\ind_{M^0}^M(\rho^0), \pi_{\bar{N}}  ] \cong
      \Hom_{M^0} [\rho^0, \pi_{\bar{N}}].\] 

      By Remark \ref{6.2},  $\pi_{\bar{N}}$ as a module of $M$ contains $\rho$ as a sub-quotient with multiplicity one, and does not contain any unramified twist of $\rho$ which is unequal to $\rho$. It easily follows
      from this that  $\dim \Hom_{M^0} [\rho^0, \pi_{\bar{N}}] \leq 1.$ 
      Therefore any irreducible representation $\pi$ of $\G_n(F)$ which appears as a quotient  of $\Ind_{P^0}^{\G_n(F)}(\rho^0)$ appears  with multiplicity at most one, and appears with multiplicity one if and only if it belongs to the Bernstein block $[M,\rho]$, and is
      a full principal series. Further, because of this multiplicity 1, the Azumaya algebra appearing in Theorem \ref{Roche}, is the ring $R$ of Laurent polynomials
      $R=\CC[X_1,X_1^{-1}, \cdots , X_d, X_d^{-1}]$
      which is the ring of regular functions on a $d$-dimensional complex torus.

      Now we analyse $\pi_1|_{\G_n(F)}[M,\rho]$ considered as a module, call it ${\mathcal M}$  over the ring,
      $R=\CC[X_1,X_1^{-1}, \cdots , X_d, X_d^{-1}]$. In the category of
      $R$-modules, irreducible = simple modules are of the form $R/{\mathfrak m}$ where ${\mathfrak m}$ are maximal
      ideals in $R$, and therefore irreducible quotients of $\pi_1|_{\G_n(F)}[M,\rho]$ are homomorphism of $R$-modules
      ${\mathcal M} \rightarrow R/{\mathfrak m} =\CC$, equivalently, homomorphism of $R$-modules
      ${\mathcal M}/{\mathfrak m}{\mathcal M}  \rightarrow \CC$. 
      As every irreducible representation in this Bernstein block is a full principal series,
      in particular they are all generic, therefore by Theorem \ref{gln} in the case of $\GL_{n+1}(F)$ and by GGP conjectures (theorems!) in other cases, each have
      these irreducible principal series representations arise as a quotient of $\pi_1$ with multiplicity 1 (multiplicity identically zero is a possibility too for $\SO_{n+1}(F),\U_{n+1}(F)$; the important thing is that the multiplicity
      is constant among all irreducible representations in this block -- which are all full principal series
      representations induced from a cuspidal representation of a Levi subgroup, a case of the GGP conjectures which is much easier than the general case).
This analysis can then be summarized to say that for the module ${\mathcal M}$  over  $R=\CC[X_1,X_1^{-1}, \cdots , X_d, X_d^{-1}]$ corresponding to $\pi_1|_{\G_n(F)}[M,\rho]$,  ${\mathcal M}/{\mathfrak m}{\mathcal M}  \cong \CC$ for all maximal ideals ${\mathfrak m}$ in $R$.

By
Theorem \ref{AS} of Aizenbud-Sayag we know the  finite generation of ${\mathcal M}$  over  $R=\CC[X_1,X_1^{-1}, \cdots , X_d, X_d^{-1}]$. Thus all the assumptions in the Lemma \ref{commutative} below are satisfied, and hence
  ${\mathcal M}$ is a  projective module of rank 1 over
      the ring $R$ of Laurent polynomials
      $R=\CC[X_1,X_1^{-1}, \cdots , X_n, X_n^{-1}]$. This is also the case for the universal principal series
  representation $\ind_{P^0}^{\G_n}(\rho^0)$. Since
  any rank 1 projective module over a Laurent polynomial ring is free,
  this concludes  the proof of the proposition.
\end{proof}

    \begin{lemma} \label{commutative}  Let $R$ be a finitely generated commutative $k$-algebra where $k$ is a field.
      Suppose that $R$ has no nonzero nilpotent
      elements. Let ${\mathcal M}$ be a finitely generated module over $R$ such that for each maximal ideal
      ${\mathfrak m}$ of $R$, ${\mathcal M}/{\mathfrak m}{\mathcal M}$ is free of rank 1 over $R/{\mathfrak m}$, then ${\mathcal M}$ is a
      projective module of rank 1
  over $R$. \end{lemma}

    \begin{remark}
      Proposition \ref{restriction} applies, in particular, to
      Bernstein blocks of $\GL_n(F)$ which contain a
      cuspidal representation of $\GL_n(F)$!
      \end{remark}

\section{Euler-Poincar\'e characteristic for classical groups}

    In the next few sections we will discuss Euler-Poincar\'e characteristic for branching laws for
    classical groups, restricting ourselves to the case of $G=\SO_{n+1}(F)$ and $H=\SO_n(F)$ since  the following variant
    of Theorem 15.1 of \cite{GGP} turns any question on Hom and Ext spaces involving
    the Bessel model to one involving  $(G,H)=(\SO_{n+1}(F), \SO_n(F))$.

    \begin{prop} \label {basic-case}
      Let $W\subset V \subset V^+ $ be non-degenerate quadratic spaces with $V/W$ an odd dimensional
      split quadratic space of dimension $(2d-1)$, and $V^+/W$ split quadratic space of dimension $2d$.
      Write $V^+= X + W + Y$ with $X,Y$ isotropic subspaces of $V^+$ of dimension $d$ with $X+Y$ nondegenerate,
      and perpendicular to $W$. Let $\pi_1$ be any finite length representation of $\SO(V)$, and $\pi_2$ of $\SO(W)$.
      Then there exists a choice
      of a unitary cuspidal representation $\tau$ of $\GL(X)$ such that for the principal series representation
      $\tau \times \pi_2$ of $\SO(V^+)$, we have:
      \[ \Ext^i_{\SO(V)}[\tau \times \pi_2, \pi_1^\vee] \cong \Ext^i_{\Bes(V,W)}[\pi_1,\pi_2^\vee] {\rm ~~for ~~all~~ i \geq 0} .\]
      \end{prop}

In the following theorem, so as to simplify notation, if $\lambda_1$ is a representation of $\SO(V_1)$
and $\lambda_2$ is a representation of $\SO(V_2)$, then by $\EP_{\Bes}[\lambda_1, \lambda_2]$ we will give it the
usual meaning (discussed in section 4) if $V_2 \subset V_1$ with $V_1/V_2$ odd dimensional split quadratic space, whereas 
$\EP_{\Bes}[\lambda_1, \lambda_2]$
will stand for
$\EP_{\Bes}[\lambda_2, \lambda_1]$ if $V_1 \subset V_2$ with $V_2/V_1$ odd dimensional split quadratic space.
The notation  $\EP_{\Bes}[\lambda_1, \lambda_2]$ will presume that we are in one of the two cases. This notation
has the utility of being able to add hyperbolic spaces of arbitrary dimension to $V_1$ or $V_2$, and by Theorem
15.1 of \cite{GGP},
\[ \EP_{\Bes}[\lambda_1, \lambda_2] = \EP_{\Bes}[\tau_1 \times \lambda_1, \tau_2 \times \lambda_2],\]
where $\tau_1,\tau_2$ are any cuspidal representations on general linear groups of arbitrary dimensions. (This result will
be extended in Theorem \ref{EP} below for all representations  $\tau_1,\tau_2$ of finite length
on general linear groups.)

The following theorem is the analogue of Theorem \ref{whittaker}
which was for representations of the groups $\GL_{n+1}(F),\GL_n(F)$, now for classical groups, but as in the rest of the paper, we assert it only for special orthogonal groups.

\begin{thm} \label{EP} Let $V$ be a quadratic space over $F$, $V'$ a nondegenerate subspace of
  codimension 1 inside $V$.
  Let $\sigma = \pi_0 \times \sigma_0,$
  be a representation for $\SO(V)$
  where $\pi_0$ is a finite length  representation of $\GL_{n_0}(F)$, and
  $\sigma_0$ is a finite length  representation of $\SO(V_0)$ where $V_0\subset V$
  is a quadratic subspace of $V$ such that the quadratic space $V/V_0$ is a
  hyperbolic space of dimension $2n_0$. Similarly, let 
  $\sigma' = \pi'_0 \times \sigma'_0$
be a representation for $\SO(V')$, then:
    \[\EP_{\SO(V')}[\sigma,\sigma'] = \dim \Wh(\pi_0) \cdot \dim \Wh(\pi_0') \cdot \dim \EP_{\Bes}[\sigma_0, \sigma'_0].\]
    \end{thm}

\begin{proof}
  The proof of this theorem is analogous to the proof of Theorem \ref{whittaker} for representations $\pi_1,\pi_2$ of $\GL_{n+1}(F),\GL_n(F)$,
  replacing the Bernstein-Zelevinsky exact sequence describing  the restriction of the representation $\pi_1$
  to the mirabolic subgroup in $\GL_{n+1}(F)$, by a similar exact sequence  describing  the restriction of the representation $\sigma$ of $\SO(V)$
  to the subgroup $\SO(V')$ due to Moeglin-Waldspurger in \cite{Mo-Wa}. An essential part of the proof of Theorem \ref{whittaker} was Lemma \ref{vanishing}
  about the vanishing of $\EP[V,W]$ when $V,W$ are finite length representations of $\GL_m(F)$, $m \geq 1$. This continues
  to be the case here. We give some details of the proof here.

  According to \cite{Mo-Wa}, the restriction of  the representation
  $\sigma = \pi_0 \times \sigma_0$ of $\SO(V)$ to $\SO(V')$ has a filtration
  with one sub-quotient equal to
  \[ \Ind_{P'}^{\SO(V')} ( \nu^{1/2} \pi_0 \times \sigma_0|_{\SO(V'_{0})}), \tag{1} \]
  where $V'_{0} =  V_0 \cap V'$ is a codimension one subspace of $V_0$, and $P'$ the parabolic in $\SO(V')$
  with Levi  $\GL_{n_0}(F) \times \SO(V'_{0}) $. (If there is no such parabolic in $\SO(V')$, then this term will not be there.)

  The other subquotients of $\sigma|_{\SO(V')}$ are the principal series representations of $\SO(V')$ induced from
  maximal parabolics with Levi whose $\GL$ part is $\GL_{n_0-i}(F)$
 and the $\SO$ part is $\SO(V_i')$
  with $V' = V'_{n_0} \supset V'_{i+1} \supset V'_i \supset V_0$
  with $V_i'/V_0$ split quadratic space of dimension $(2i-1)$, in particular, $\dim V' = \dim V_0 +2n_0-1$. These
  subquotients are:

  \[ \pi_0^i \times \ind_{\Bes(V_0)}^{\SO(V_i')} (\sigma_0 \otimes \psi)
  , \,\,\,\,\,\, n_0\geq i \geq 1, \tag{2} \]
  where $\pi_0^i$ is a representation of $\GL_{n_0-i}(F)$ which is the $i$-th derivative of the
  representation $\pi_0$ of $\GL_{n_0}(F)$.

  Given the filtration on $\sigma|_{\SO(V')}$ with successive quotients as in (1) and (2),
  and as $\EP[\sigma, \sigma']$ is additive in exact sequences,
  one applies the Ext version of the 2nd adjointness theorem of Bernstein,
  cf. Theorem 2.1 of \cite{Pr3}, to find that 
  $\EP[\sigma, \sigma']$ is a sum of terms corresponding to the above filtration on $\sigma$ restricted to
  $\SO(V')$ which,  for the term in equation (1) above,  is:
  \[ \EP_{\GL_{n_0}(F) \times \SO(V'_{0})}[ \nu^{1/2}\pi_0 \times \sigma_0|_{\SO(V'_{0})} , \sigma'_{N'^{-}}]  \tag{3}\]
  where $N'^{-}$ is the unipotent radical of the parabolic opposite to $P'$,
  and the subscript to $\sigma'$ denotes
  the corresponding Jacquet module.

  The contribution of the terms in equation (2) to  $\EP[\sigma, \sigma']$ is:
  \[ \EP_{\GL_{n_0-i}(F) \times \SO(V'_{i})}[ \pi_0^i \times \ind_{\Bes(V_0)}^{\SO(V_i')} (\sigma_0 \otimes \psi), \sigma'_{N_i'^{-}}] \tag{4}, \]
  where $N'^{-}_i$ is the unipotent radical of the parabolic opposite to $P'_i$.

  As $\sigma'$ is a finite length representation of $\SO(V')$, its Jacquet modules  $\sigma'_{N'^{-}}$, and  $\sigma'_{N_i'^{-}}$
  are finite length representations of the corresponding Levi subgroups of $\SO(V')$ which are of the form $\GL_{n_0}(F) \times \SO(V'_{0})$
  and $\GL_{n_0-i}(F) \times \SO(V'_{i})$ respectively, and therefore in the Grothendieck group of representations of the
  corresponding Levi subgroups,
  these Jacquet modules are sum of tensor
  products of irreducible representations of the two factors, one of the general linear group and one of the special orthogonal group. We are thus in the case (1)  of Theorem \ref{kunneth} (where $G_1$, in the notation of that theorem, is a general linear group, and the representations  $E_1,F_1$ of $G_1$ are of finite length).

  By  Lemma \ref{vanishing}
  about vanishing of $\EP[V,W]$ when $V,W$ are finite length representations of $\GL_m(F)$, $m \geq 1$, and
  Corollary \ref{kunneth-cor} about Euler-Poincar\'e pairing for product of groups, it follows that the only term which gives a nonzero contribution to    $\EP[\sigma, \sigma']$
  is the one from the filtration of $\sigma$ restricted to $\SO(V')$ coming from (2) and corresponding to  
  the highest derivative $\pi_0^{n_0}$ of $\pi_0$ (note that $V'_{n_0}=V'$) giving rise to the representation
  of $\GL(0) = \{1\}$ of dimension $\dim \Wh(\pi_0)$ which needs to be multiplied by
  \[\EP [\ind_{\Bes(V_0)}^{\SO(V')} (\sigma_0 \otimes \psi), \sigma'] = \EP_{\Bes(V_0)}[\sigma'^\vee, \sigma_0^\vee], \tag{5}\]
  which is the term in (4) corresponding to $i=n_0$, in which case the corresponding
  Jacquet module $\sigma'_{N_i'^{-}}$ is simply $\sigma'$. (The equality in (5) is just the Frobenius reciprocity.)

  Thus, we have proved that
  \[\EP_{\SO(V')}[\sigma,\sigma'] = \dim \Wh(\pi_0) \cdot \dim \EP_{\Bes(V_0)}[\sigma'^\vee,\sigma_0^\vee]. \tag{6}\]
  
  One can carry out the above analysis once more, using now the representation   $\sigma' = \pi'_0 \times \sigma'_0,$ and  we get,
    \[ \EP_{\Bes(V_0)}[\sigma'^\vee,\sigma_0^\vee] = \dim \Wh(\pi'_0) \cdot \dim \EP_{\Bes(V'_0)}[\sigma_0,\sigma'_0]. \tag{7}\]

  By (6) and (7), we find:
  \[\EP_{\SO(V')}[\sigma,\sigma'] = \dim \Wh(\pi_0) \cdot \dim \Wh(\pi'_0) \cdot \dim
  \EP_{\Bes}[\sigma_0,\sigma'_0], \tag{8}\]
  completing the proof of Theorem \ref{EP}.
  \end{proof}

\begin{cor} \label{constancy}
  Let,
  \[\sigma = \pi_1\times \cdots \times  \pi_t \times \sigma_0,\]
    \[\sigma' = \pi'_1 \times \cdots \pi'_{t'} \times \sigma'_0,\]
    be representations of $\SO(V),\SO(V')$, and
    \[\sigma_b = \pi_1|\cdot|_F^{b_1}\times \cdots \times  \pi_t|\cdot|_F^{b_t} \times \sigma_0,\]
    \[\sigma'_{b'} = \pi'_1|\cdot|_F^{b'_1}\times \cdots \pi'_{t'}|\cdot|_F^{b'_{t'}} \times \sigma'_0,\]
  be families of representations passing through $\sigma,\sigma'$.  Then $\EP_{\SO(V')}[\sigma_b,\sigma'_{b'}]$ is independent of $b,b'$.
  \end{cor}
\begin{proof} By Theorem \ref{EP}, it suffices to note that  $\dim \Wh(\pi_b)$ and
  $ \dim \Wh(\pi'_{b'})$ are constant which follows since dimension of the space of
  Whittaker models of a parabolically  induced representation is the same as that of the inducing representation. 
  \end{proof}
\section{Local constancy of the Euler-Poincar\'e characteristic}
In corollary \ref{constancy} we proved local constancy of the Euler-Poincar\'e characteristic. In this section,
we offer a different and a more direct proof of a part of the local constancy which does
not use the detailed analysis via the Bernstein-Zelevinski like filtration. There is a proof of a much more general
case due to Aizenbud-Sayag in \cite{AS}.

\begin{prop} Let $(G,H)$ be a pair of $p$-adic groups, with $H$ reductive, and
  $\pi_1,\pi_2$ smooth representations of $G, H$ respectively.  Assume that
  the restriction of the representation $\pi_1$ of $G$ to $H$ is locally finitely generated (this is the case by results of
  \cite{AS} for the GGP pairs). Suppose that $\pi_2$ is a finite length representation of $H$ which is part of a family $\pi_{2, \lambda}$
  of representations of $H$, i.e., $\pi_2$ is a principal series representation of $H$ of the form
  $\pi_2 = \Ind_P^H(V_2)$ where $P$ is a parabolic in $H$ with Levi decomposition $P=MN$, and $V_2$ is a finite length
  representation of $M$, and  $\pi_{2, \lambda} = \Ind_P^H(V_2 \otimes \lambda)$ where $\lambda$ is an unramified character of $M$. Then,
  \[ \EP_H[\pi_1,\pi_{2, \lambda}],\]
    is independent of $\lambda$.
\end{prop}

\begin{proof}
  Let $\pi_1[\pi_2]$ be the projection of $\
  pi_1$ in the sum of those finitely many
  blocks of $H$, a reductive $p$-adic group,  in which $\pi_2$ has a nontrivial component. 
  As $\pi_1$ is a locally finitely generated representation of $H$, 
  $\pi_2$ has finite length,  $\pi_1[\pi_2]$ is a finitely generated $H$-module, and therefore
  has a finite resolution by finitely generated
  $H$-projective modules. This is a generality which uses:

  \begin{enumerate}
  \item The category of smooth representations of $H$ is noetherian, i.e., submodule of a finitely generated module is
    finitely generated.
  \item The category of smooth representations of $H$ has finite homological dimension (equal to the split rank of $H$ assumed to be $n$),
    i.e., every smooth representation of $H$ has a finite projective resolution of length $n$
    (by in general infinitely generated modules).
  \end{enumerate}

  By these two generalities, it follows that  $\pi_1[\pi_2]$ has a finite resolution by finitely generated
  projective modules $P_i$:
\[ 0\rightarrow Q_n \rightarrow P_{n-1} \cdots \rightarrow P_1 \rightarrow P_0 \rightarrow \pi_1[\pi_2] \rightarrow 0,\]   
where by the usual dimension shift argument, $Q_n=P_n$ is a projective module (where $n$ is the cohomological
dimension of the category of smooth representations of $H$).

Now, $\Ext^i_H[\pi_1,\pi_{2, \lambda}] =\Ext^i_H[\pi_1[\pi_2] ,\pi_{2, \lambda}]$
is the cohomology of the  complex consisting of
$\Hom_H[P_i,\pi_{2, \lambda}]$. We claim that $\Hom_H[P_i,\pi_{2, \lambda}]$ are finite dimensional
complex vector spaces whose dimension is independent of $\lambda$.

Since $P_i$ are finitely generated, by the noetherian property of the category of smooth representations
of $H$, $P_i$ have a resolution of the form:
\[ \ind_{K'}^H(V_i') \rightarrow \ind_K^H(V_i)  \rightarrow P_i \rightarrow 0,\]   
where $K,K'$ are compact open subgroups of $H$, and $V_i,V'_i$ are finite dimensional
smooth representations of $K,K'$. Therefore, applying $\Hom_G[-, \pi_{2,\lambda}]$ to the above exact sequence, we get an exact
sequence:
\[0 \rightarrow  \Hom_G[P_i, \pi_{2,\lambda}] \rightarrow \Hom_G[\ind_{K}^H(V_i), \pi_{2,\lambda}]  \rightarrow
\Hom_G[\ind_{K'}^H(V'_i), \pi_{2,\lambda}]  \]
By Frobenius reciprocity, Lemma \ref{Frob},
this is same as
\[ 0 \rightarrow \Hom_G[P_i, \pi_{2,\lambda}] \rightarrow \Hom_K[V_i, \pi_{2,\lambda}]  \rightarrow
\Hom_{K'}[V'_i, \pi_{2,\lambda}] . \]

Now note that
since $\pi_{2, \lambda}$ is a principal series representation of $H$ with $\lambda$ an unramified character
of the Levi subgroup used in the induction, $\pi_{2, \lambda}|_K$ is independent of $\lambda$,
 hence $\Hom_K[V_i ,\pi_{2, \lambda}]$ are independent of $\lambda$, and by
 admissibility of $\pi_{2, \lambda}$,  $\Hom_K[V_i ,\pi_{2, \lambda}]$ are finite dimensional $\CC$-vector space.
 Similarly for the finite dimensional representations $V_i'$ of $K'$.
Thus, $\Hom_K[V_i, \pi_{2,\lambda}]$ and   
$\Hom_{K'}[V'_i, \pi_{2,\lambda}]$ are   complex vector spaces whose dimension is independent of $\lambda$,
  and the map too between them is independent of $\lambda$, therefore the dimension
  of the kernel, i.e.,  $\dim \Hom_G[P_i, \pi_{2,\lambda}]$ is independent of $\lambda$.

The proof of the proposition is completed by the well-known assertion that the Euler-Poincar\'e characteristic
of a finite complex of finite dimensional vector spaces is the alternating sum of their dimensions
which is independent of $\lambda$.
\end{proof}

\section{Euler-Poincar\'e characteristic  for the group case: Kazhdan orthogonality}

The branching laws considered in this paper are for $H \hookrightarrow G \times H$,
where $H \subset G$, eventually interpreted as the $(G \times H)$ spherical variety
$\Delta(H) \backslash (G \times H)$,
in which we try to understand $\Ext^i_H(\pi_1,\pi_2)$
for an irreducible representation
$\pi_1 \boxtimes \pi_2$ of $G \times H$. A special case of this branching problem  is for  the ``group case''
where $H=G$, so in the case of
$G \hookrightarrow G \times G$ where we will be considering  $\Ext^i_G(\pi_1,\pi_2)$ where $\pi_1, \pi_2$
are representations of the same group $G$. This could be considered as a precursor of the more general branching
for $H \hookrightarrow G \times H$, and has played an important  role in the subject.

Explicit calculation of $\Ext_G^i(\pi_1,\pi_2)$ has been carried out in several cases for $G$, a reductive $p$-adic group, and $\pi_1,\pi_2$
irreducible representations of $G$. In particular, if both $\pi_1$ and $\pi_2$ are tempered representations of $G$,
there are general results in \cite{Op-Sol} using the formulation of $R$-groups, and there are some independent
specific calculations in \cite{Ad-Dp}.
Ext groups for certain non-tempered representations are considered in \cite{Or} as well as in \cite{Dat}.

In the archimedean case, $H^i({\mathfrak g}, K, \pi) = \Ext^i(\CC, \pi)$ has been much studied, but not $\Ext^i(\pi_1,\pi_2)$ as far as I know.

The following theorem was conjectured by Kazhdan and was proved by  Schneider-Stuhler, cf. \cite{Sch-Stu},  and by Bezrukavnikov. It is  known only in characteristic zero.
\begin{theorem}
Let $\pi$ and $\pi'$ be finite-length, smooth
representations of a
reductive
$p$-adic
group $G$.
Then \[\EP_G[\pi,\pi'] = 
\int_{C_{ellip}} \Theta(c)\bar{\Theta}'(c)\, dc,\] 
where $\Theta$ and $\Theta'$ are the characters of $\pi$ and $\pi'$, 
and $dc$ is a natural measure on the set
$C_{ellip}$ of regular elliptic conjugacy classes in $G$, and is given by
\[ dc =  {W(G(F), T(F))}^{-1}  \cdot \| \det (1- Ad(\gamma))_{{\frak g}/{\frak g}_\gamma} \| dt,\]
where $dt$ is the normalized Haar measure on the elliptic torus $T= G_\gamma$ giving it measure 1.
\end{theorem}

\section{An integral formula of Waldspurger}

In this section we review an integral formula of Waldspurger, cf. \cite{Wa}, \cite{Wa1},
which we then propose to be the 
integral formula for the Euler-Poincar\'e pairing for 
\[\EP_{\Bes(V,W)}[\sigma,\sigma' ]\] 
for $\sigma$ 
any finite length
representation of $ \SO(V)$, and $\sigma'$ any finite length representation of $\SO(W)$, 
where $V$ and $W$ are quadratic spaces over $F$ with 
\[V = X + D + W + Y\]
with $W$ a quadratic subspace of $V$ of codimension $2k+1$ 
with $X$ and $Y$ totally isotropic subspaces of $V$ of dimension $k$, in duality with each other under the underlying bilinear form,  
and $D$ an anisotropic line in $V$. Let $Z=X+Y$.

Let $\underline{\mathcal T}$ denote the set of elliptic tori $T$ in $\SO(W)$ such that there exist quadratic subspaces  $W_T,W'_T$ of $W$ such that:  
\begin{enumerate}

\item $W= W_T \oplus W'_T$, and $V=W_T \oplus W'_T \oplus D \oplus Z$.

\item $\dim (W_T)$ is even, and $\SO(W'_T)$ and 
$\SO(W'_T \oplus D \oplus Z)$ are  quasi-split. 

\item $T$ is a maximal (elliptic) torus in $\SO(W_T)$.

\end{enumerate}

Let
${\mathcal T}$ denote a set of orbits for the action of $\SO(W)$  on $\underline{\mathcal T}$. 
For our purposes we note the
most important elliptic torus $T= \langle e \rangle$ 
corresponding to $W_T= 0$.

For $\sigma$ an admissible representation of $\SO(V)$ of finite length, define a function 
$c_\sigma(t)$ for regular elements of a  torus $T$ belonging to $\underline{\mathcal T}$ by the germ
expansion of the character $\theta_\sigma(t)$ of $\sigma$ on the centralizer of $t$ in the Lie algebra 
of $\SO(V)$, and picking out `the' leading term.

Similarly, for $\sigma'$ an admissible representation of $\SO(W)$ of finite length, 
one defines a function 
$c_{\sigma'}(t)$ for regular elements of a  torus $T$ belonging to $\underline{\mathcal T}$ by the germ
expansion of the character $\theta_{\sigma'}(t)$ of $\sigma'$.

Define a function $\Delta_T$ on an elliptic torus $T$ belonging to
$\underline{\mathcal T}$ with $W= W_T \oplus W'_T$, by 
\[\Delta(t) = |\det(1-t)|_{W_T}|,\] and let  $D^H$ denote the function 
on $H(F) = \SO(W)$ defined by:
\[ D^H(t) = |\det(Ad(t) -1)_{h(F)/h_t(F)}|_F,\]
where $h(F)$ is the Lie algebra of $H$ and $h_t(F)$ is the Lie algebra of the centralizer of $t$ in $H$.

For a torus $T$ in $H$, define the Weyl group $W(H,T)$ by the usual normalizer divided by the centralizer:
\[W(H,T) = N_{H(F)}(T)/Z_{H(F)}(T).\]

  The following theorem of Waldspurger could be considered as the analogue
  of Kazhdan orthogonality for the group case encountered earlier.

\begin{thm} \label{Wal}
Let $V = X + D + W + Y$ be a quadratic space over the non-archimedean local field $F$ with $W$ a quadratic subspace of codimension $2k+1$ as above. 
Then for any  irreducible admissible representation $\sigma$ of $\SO(V)$ 
and irreducible admissible representation $\sigma'$ of $\SO(W)$,  
\[c(\sigma,\sigma'): =\sum _{T \in {\mathcal T}} |W(H,T)|^{-1} \int_{T(F)} c_\sigma(t) c_{\sigma'}(t) D^H(t) \Delta^k(t) dt ,\] 
is a finite sum of absolutely convergent integrals. (The Haar measure on $T(F)$ is normalized to have volume 1.) If either $\sigma$ is a supercuspidal representation of 
$\SO(V)$, 
and $\sigma'$ is arbitrary irreducible admissible representation of $\SO(W)$,
or both $\sigma$ and $\sigma'$ are tempered representations, then 
$$c(\sigma,\sigma') = \dim \Hom_{\Bes(V,W)}[\sigma,\sigma' ].$$
\end{thm}
\section{Conjectured EP formula}

Given the Theorem \ref{Wal} of Waldspurger, it is most natural to propose the following conjecture on
Euler-Poincar\'e pairing following the earlier notation of
$V = X + D + W + Y$, a quadratic space over the non-archimedean local field $F$ with $W$ a quadratic subspace of 
$V$ of codimension $2k+1$. 

\begin{conj} \label{integral}
  \begin{enumerate}

\item   If $\sigma$ and $\sigma'$ are  irreducible tempered representations of $\SO(V), \SO(W)$ respectively with
  $W\subset V$, a nondegenerate subspace  with $V/W$ a split quadratic space of odd dimension, then
\[\Ext^i_{\Bes(V,W)}[\sigma,\sigma'] = 0 {\rm~~ for~~} i > 0.\]

\item   For finite length representations $\sigma$ of $\SO(V)$ 
and  $\sigma'$ of $\SO(W)$, we have:  
  \begin{eqnarray*}
 \EP_{\Bes(V,W)}[\sigma, \sigma' ] & : =&   
  \sum_i (-1)^i \dim \Ext^i_{\Bes(V,W)}[\sigma,\sigma'], \\ 
  & = & \sum _{T \in {\mathcal T}} |W(H,T)|^{-1} \int_{T(F)} c_\sigma(t) c_{\sigma'}(t) D^H(t) \Delta^k(t) dt.
  \end{eqnarray*}
\end{enumerate}

\end{conj}

\begin{remark}
  \begin{itemize}
    \item Waldspurger integral formula is available also in the work of R. Beuzart-Plessis for unitary groups.

\item For both the assertions of Conjecture \ref{integral}, 
  we are reduced to the ``basic case'' when
  $\dim(V/W)=1$ by Proposition \ref{basic-case}.

  \item  A general integral formula  for spherical varieties has been formulated by Chen Wan in \cite{Wan}.
\end{itemize}
\end{remark}

In what follows,  let
  \[\sigma = \pi_1|\cdot|_F^{b_1}\times \cdots \times  \pi_t|\cdot|_F^{b_t} \times \sigma_0,\]
  be a standard module for $\SO(V)$, thus, we have:
  \begin{enumerate}
  \item For $i=1,\cdots, t$, $\pi_i$ is an irreducible, admissible, tempered representation of $\GL_{n_i}(F)$.
  \item  $\sigma_0$ is an irreducible, admissible, tempered representation of $\SO(V_0)$ where $V_0\subset V$
    is a quadratic subspace of $V$ such that the quadratic space $V/V_0$ is an orthogonal direct sum
    of hyperbolic spaces of dimensions $2n_i$.
  \item The $b_i$ are real numbers with $b_1\geq b_2\geq \cdots \geq b_t \geq 0$.
  \end{enumerate}

  Similarly, let
  \[\sigma' = \pi'_1|\cdot|_F^{b'_1}\times \cdots \times  \pi'_{t'}|\cdot|_F^{b'_{t'}} \times \sigma'_0,\]
  be a standard module for $\SO(W)$.

  The following theorem is Proposition 1.1 of \cite{Mo-Wa}; its variant for unitary groups is Proposition 9.4 due to
  Gan-Ichino, cf. \cite{GI}, and for the general linear group, it is  due to Chan in \cite{Chan2}.
  \begin{thm} \label{MW-GI-Ch}
    With the notation as above, in particular when $\sigma,\sigma'$ are standard modules
    of $\SO(V),\SO(W)$ respectively, and denoting $\check{\sigma}'$, the dual of $\sigma'$ (following the notation of
\cite{Mo-Wa} momentarily), we have:
     \[ \dim \Hom_{\SO(W)}[\sigma, \check{\sigma}'] = \dim \Hom_{\SO(W_0)}[\sigma_0, \sigma_0'] = \dim \Hom_{\SO(W_0)}[\sigma_0, \check{\sigma}_0'] .\]
    \end{thm}

  \begin{remark}
    A consequence of Theorem \ref{MW-GI-Ch} is an extension of the multiplicity one theorem of
    A. Aizenbud, D. Gourevitch, S. Rallis and G. Schiffmann in \cite{AGRS} to standard modules
    which are not necessarily irreducible, but notice that the theorem above is only about $\Hom_{\SO(W)}[\sigma, \check{\sigma}']$ 
    when $\sigma, \sigma'$ are standard modules of $\SO(V),\SO(W)$ respectively, and not for any two general
    principal series representations (for which it is false).
  \end{remark}

  \begin{cor} \label{vanishing}
With the notation as above, in particular when $\sigma,\sigma'$ are standard modules
    of $\SO(V),\SO(W)$ respectively, and denoting $\check{\sigma}'$, the dual of $\sigma'$, we have:
     \[ \dim \Ext^i_{\SO(W)}[\sigma, \check{\sigma}'] \leq  \dim \Ext^i_{\SO(W_0)}[\sigma_0, \sigma_0'] {\rm ~~ for~~all~~} i\geq 0.\]
     In particular,  if  $\Ext^i_{\SO(W_0)}[\sigma_0, \sigma_0']=0 $, for all $i>0$,  then  $\Ext^i_{\SO(W)}[\sigma, \check{\sigma}']=0$
     also for all $i>0$.
  \end{cor}

  \begin{proof} The proof of this corollary is verbatim the same as the proof of Theorem \ref{MW-GI-Ch} which uses
    the Bernstein-Zelevinski filtration on representations of classical groups, as discussed during the proof of Theorem \ref{EP}; the
    proof eventually compares exponents of representations of $\GL_d(F)$ appearing on the two sides of the Hom spaces
    to say that they must be zero for all but one term. This argument works equally well for the conclusion of the Ext
    spaces. \end{proof}

The following theorem asserts that once  $\Ext_{\Bes}^i[\pi_1,\pi_2]$ are known to be zero for tempered
representations for $i \geq 1$, the Conjecture \ref{integral} on  Waldspurger integral formula giving
Euler-Poincar\'e characteristic for all finite length representations  holds.

\begin{thm} \label{conj-gen}
Assume vanishing of higher Ext groups for tempered representations, i.e.,
  for an irreducible tempered representation
  $\sigma$ of $\SO(V)$ 
and  $\sigma'$ of $\SO(W)$, assume that,
\[\Ext^i_{\Bes(V,W)}[\sigma,\sigma'] = 0  {\rm ~~for~~} i > 0.\]
Then for  any finite length representation $\sigma$ of $\SO(V)$ 
and  any finite length representation $\sigma'$ of $\SO(W)$, we have the
Euler-Poincar\'e formula:
\begin{eqnarray*} \EP_{\Bes(V,W)}[\sigma, \sigma' ] & : = &   
  \sum_i (-1)^i \dim \Ext^i_{\Bes(V,W)}[\sigma,\sigma'] \\ &=&
  \sum _{T \in {\mathcal T}} |W(H,T)|^{-1} \int_{T(F)} c_\sigma(t) c_{\sigma'}(t) D^H(t) \Delta^k(t) dt.
  \end{eqnarray*}
  \end{thm}
\begin{proof}
  Note that  both sides of the proposed equality:
  \begin{eqnarray*} \EP_{\Bes(V,W)}[\sigma, \sigma' ] & = &    \sum _{T \in {\mathcal T}} |W(H,T)|^{-1} \int_{T(F)} c_\sigma(t) c_{\sigma'}(t) D^H(t) \Delta^k(t) dt,
  \end{eqnarray*}
  are bilinear forms for $\sigma$ belonging to the Grothendieck group of
  finite length representations of $\SO(V)$, and $\sigma'$ belonging to the Grothendieck group of
  finite length representations of $\SO(W)$. It, therefore suffices to prove it for
  $\sigma$ belonging to a set of generators for
  the Grothendieck group of
  finite length representations of $\SO(V)$, and $\sigma'$ belonging to a set of generators
  for the Grothendieck group of
  finite length representations of $\SO(W)$. It is well known that
  the standard modules form a set of generators, in fact a basis, of the Grothendieck group of finite length representations
  of any reductive $p$-adic group. Therefore if we can prove:
  \begin{eqnarray*} \EP_{\Bes(V,W)}[\sigma, \sigma' ] &  = &   
  \sum _{T \in {\mathcal T}} |W(H,T)|^{-1} \int_{T(F)} c_\sigma(t) c_{\sigma'}(t) D^H(t) \Delta^k(t) dt,
  \end{eqnarray*}
  for $\sigma, \sigma'$ standard modules, we would know it for all $\sigma,\sigma'$ finite length modules of $\SO(V),\SO(W)$ respectively.

  Under the assumption made for Theorem \ref{conj-gen} about vanishing of higher Ext groups for tempered representations, by Corollary \ref{vanishing},
  $\Ext^i_{\Bes(V,W)}[\sigma,\sigma']=0$ for $i>0$ for all $\sigma,\sigma'$ standard modules of $\SO(V),\SO(W)$ respectively. Therefore,
   \[ \EP_{\Bes(V,W)}[\sigma, \sigma' ] = \dim \Hom_{\Bes(V,W)}[\sigma,\sigma']= \dim  \Hom_{\Bes(V_0,W_0)}[\sigma_0,\sigma_0'],\]
   where the last equality is Theorem \ref{MW-GI-Ch} due to Moeglin-Waldspurger for orthogonal groups extended
   to unitary groups by Gan-Ichino,  and is due to Chan for the general linear group.

It  then suffices to prove that the sum:
  \[  \sum _{T \in {\mathcal T}} |W(H,T)|^{-1} \int_{T(F)} c_\sigma(t) c_{\sigma'}(t) D^H(t) \Delta^k(t) dt,\]
  is the same for the pair $(\sigma_0, \sigma'_0)$ as it is for the pair $(\sigma, \sigma')$. This is a consequence of  the
  van Dijk formula for principal series representations, see Lemma 2.3 of \cite{Wa1}. \end{proof}

\section{The Schneider-Stuhler duality theorem} 

The following theorem is a mild generalization of a
duality theorem due to  Schneider and Stuhler in \cite{Sch-Stu}, see \cite{NP}; it turns
questions on $\Ext^i[\pi_1,\pi_2] $ to $\Ext^j[\pi_2,\pi_1]$, and is of considerable
importance to our theme.

\begin{thm} \label{SS}
  Let $G$ be a reductive $p$-adic group, and $\pi$ an irreducible admissible representation of $G$. 
  Let $d(\pi) $ be the split rank of the center of the Levi subgroup $M$ of $G$ which carries the cuspidal support of $\pi$,
  $D(\pi)$ be the Aubert-Zelevinsky involution of $\pi$. Then,
     
\begin{enumerate}

\item $\Ext_G^{d(\pi)}[\pi, D(\pi)] \cong \CC$, and  
\item For any smooth  representation $\pi'$ of $G$, the bilinear pairing
\[(*) \,\,\,\,\,   \Ext^{i}_G[\pi, \pi']  \times \Ext^{j}_G[\pi', D(\pi)]  \rightarrow 
\Ext^{i+j = d(\pi)}_G[\pi, D(\pi)]   \cong \CC, \]
is non-degenerate.
\end{enumerate}
\end{thm}

\section{An example: triple products for $\GL_2(F)$}

As suggested earlier, we expect that for all the GGP pairs $(G,H)$, when the irreducible representation
$\pi_1$ of $G$ and $\pi_2$ of $H$ are tempered,
$\Ext^i_H[\pi_1,\pi_2]$  is non-zero only for $i=0$.

On the other hand, by the duality theorem discussed in the last section,
we expect that $\Ext^i_H[\pi_2,\pi_1]$ is typically zero for $i=0$, i.e., $\Hom_H[\pi_2,\pi_1]=0$ (so no wonder branching is usually not considered as a subrepresentation!), and shows up only for  $i$ equals the split rank of the center of the Levi 
from which $\pi_2$  arises through parabolic induction of a 
supercuspidal representation. This is not completely correct as we will see.

The purpose of this section is to  do an explicit  restriction problem as
an example of what happens for classical groups
in one specific instance: the restriction problem from split $\SO(4)$ to split $\SO(3)$.
Thus we  calculate $\Ext^i_{\SO_3(F)}[V,V']$, $i\geq 0$, and $\EP[V,V']$,  
for $V$ a  representation of $\SO_4(F)$ of finite length,
and $V'$ a representation of  $\SO_3(F)$ of finite length, and then investigate when the restriction of $V$ from $\SO_4(F)$ to
$\SO_3(F)$  is a projective module. As a consequence of what we do here, we will have constructed a projective module
in the Iwahori block of $\SO_3(F)=\PGL_2(F)$ which is different from what we encountered earlier in the restriction
problem from $\GL_3(F)$ to $\GL_2(F)$ which all had only generic representations as a quotient, but here there will
be another possibility. (In fact Lemma 2.4 of  \cite{CS3} has two options for projective modules which have multiplicity 1, and this other possibility which we will see here is the second  option for projective modules in  Lemma 2.4 of  \cite{CS3}.)

Since
$\SO_4(F)$ and $\SO_3(F)$ are closely related to $\GL_2(F) \times \GL_2(F)$ 
and $\GL_2(F)$ respectively, we equivalently consider $V \cong \pi_1 \otimes \pi_2$  
for admissible representations $\pi_1, \pi_2$ of $\GL_2(F)$, and 
$V' = \pi_3$ of $\GL_2(F)$. Our aim then is to calculate

$$\Ext^i_{\GL_2(F)}[\pi_1 \otimes \pi_2, \pi_3],$$ 
or since we will prefer not to bother with central characters, we assume
that $\pi_1 \otimes \pi_2$ and $\pi_3$ have trivial central characters, and we will
then calculate,
$$\Ext^i_{\PGL_2(F)}[\pi_1 \otimes \pi_2, \pi_3].$$ 

\begin{prop} \label{trilinear} Let 
$\pi_1, \pi_2$ and $\pi_3$ 
be either irreducible, infinite dimensional  representations of $\GL_2(F)$, or (reducible) 
principal series  representations 
of $\GL_2(F)$ induced from one dimensional representations. Assume that the product of the central characters  of $\pi_1$ and $\pi_2$ is trivial, and $\pi_3$ is
of trivial central character. Then,
\[ \EP_{\PGL_2(F)}[\pi_1 \otimes \pi_2, \pi_3] =1,\]
except when all the representations $\pi_i$ are irreducible discrete series representations, and there is 
a $D^\times$ invariant linear form on $\pi'_1 \otimes \pi'_2 \otimes \pi'_3$ where $\pi_i'$ denotes the 
representation of $D^\times$ associated to $\pi_i$ by Jacquet-Langlands. In the case when
$\pi'_1 \otimes \pi'_2 \otimes \pi'_3$ has a $D^\times$-invariant linear form, we have,
\[ \EP_{\PGL_2(F)}[\pi_1 \otimes \pi_2, \pi_3] =0,\]
and,
\[ \EP_{D^\times}[\pi'_1 \otimes \pi'_2, \pi'_3] =1.\]
\end{prop}

\begin{proof}
The proposition follows from more general earlier results on Euler-Poincar\'e characteristic for principal
series representations or can be deduced directly from Mackey theory, we give a brief
outline of the proof.

If at least one of the $\pi_i$ is cuspidal, then it is easy to see
that $ \Ext^1_{\PGL_2(F)}[\pi_1 \otimes \pi_2, \pi_3] =0$, and the proposition is equivalent to by-now well-known
results, cf. \cite{Pr1} about $ \Hom_{\PGL_2(F)}[\pi_1 \otimes \pi_2, \pi_3]$. The proposition in case one of the $\pi_i$'s is a twist of the Steinberg representation of $\GL_2(F)$ follows by embedding the Steinberg representation
in the corresponding principal series, and using additivity of EP in exact sequences.  \end{proof}

Since for $\PGL_2(F)$, the only nonzero $\Ext^i$ can be for $i=0,1$, $\EP$ together with knowledge of $\Hom$ spaces implies
the following corollary, a case of the very general conjecture, Conjecture \ref{integral} (1), in the special case of the pair
$(\SO_4(F),\SO_3(F))$ with $\SO_4(F)$ a split group.

\begin{cor} \label{ext-vanish} If  $\pi_1, \pi_2$ and $\pi_3$  are any three irreducible, infinite dimensional  representations 
 of $\GL_2(F)$, 
with the product of the central characters  of $\pi_1$ and $\pi_2$ trivial, and $\pi_3$ 
of trivial central character, then $\Ext^i_{\PGL_2(F)}[\pi_1 \otimes \pi_2,\pi_3] = 0$ for $i>0$.
\end{cor}

\begin{remark}
  The paper \cite{CF} of  Cai and Fan  
  proves more generally that
  \[ \Ext^i_{\GL_2(F)}[\Pi, \CC] = 0 ~~~{\rm~~ for ~~~} i \geq 1,\]  
  where $\Pi$ is an irreducible generic representation of $\GL_2(E)$, $E$ a cubic \'etale extension of $F$,
      whose central character restricted to $F^\times$ is trivial.
  \end{remark}

We now use Proposition \ref{trilinear} and Corollary \ref{ext-vanish}
to study the restriction problem from $[\GL_2(F) \times \GL_2(F)]/\Delta (F^\times)$ to $\PGL_2(F)$, and to understand when $\pi= \pi_1 \otimes \pi_2$ where $\pi_1, \pi_2$   are any two irreducible, infinite dimensional  representations 
 of $\GL_2(F)$, 
with the product of the central characters  of $\pi_1$ and $\pi_2$ trivial, is a projective representation of $\PGL_2(F)$

\begin{prop} \label{gl2} Let 
$\pi_1, \pi_2$ 
  be irreducible, infinite dimensional  representations of $\GL_2(F)$ such that the product of the central
  characters  of $\pi_1$ and $\pi_2$ is trivial. Then the representation $\pi_1\otimes \pi_2$ of $\PGL_2(F)$ is a
  projective module unless  $\pi_1,\pi_2$ are irreducible principal series representations
  of $\GL_2(F)$ such that  $\pi_1 \cong \chi \pi_2^\vee$ with $\chi$ a quadratic character, in which case it is
  not a projective module exactly in the block of $\PGL_2(F)$ containing the character $\chi$. In particular,
  if at least one of $\pi_1$ or $\pi_2$ is a twist of the Steinberg representation, $\pi_1\otimes \pi_2$ is a projective
  module in the category of smooth representations of $\PGL_2(F)$.
\end{prop}

\begin{proof}
  As discussed
earlier, if a smooth representation $\pi$ of $\PGL_2(F)$ is locally finitely generated, then it is a projective module in the category
of smooth representations of $\PGL_2(F)$ if and only if
\[ \Ext^1_{\PGL_2(F)}[\pi, \pi'] = 0, \]
for all smooth  finitely generated representations $\pi'$ of $\PGL_2(F)$ which is the case   
if and only if
\[ \Ext^1_{\PGL_2(F)}[\pi, \pi'] = 0, \]
for all finite length representations $\pi'$ of $\PGL_2(F)$, which is the case 
if and only if
\[ \Ext^1_{\PGL_2(F)}[\pi, \pi'] = 0, \]
for all irreducible representations $\pi'$ of $\PGL_2(F)$.

In our case, $\pi= \pi_1 \otimes \pi_2$ where $\pi_1, \pi_2$   are any two irreducible, infinite dimensional  representations 
 of $\GL_2(F)$, 
with the product of the central characters  of $\pi_1$ and $\pi_2$ trivial. Therefore if $\pi'$ is any  
 infinite dimensional irreducible  representation 
 of $\PGL_2(F)$, the desired vanishing of $\Ext^1[\pi,\pi']$ is  Corollary \ref{ext-vanish}. Therefore to
 prove projectivity of $\pi= \pi_1 \otimes \pi_2$ as a representation of $\PGL_2(F)$, it suffices to check that,
 \[ \Ext^1_{\PGL_2(F)}[\pi_1\otimes \pi_2, \chi] =  0, \]
 for $\chi: F^\times /F^{\times 2} \rightarrow \CC^{\times}$, treated as a character of $\PGL_2(F)$. Now,
 \[ \Ext^1_{\PGL_2(F)}[\pi_1\otimes \pi_2, \chi] \cong \Ext^1_{\GL_2(F), Z}[\pi_1, \chi \pi_2^\vee],\]
 where $\Ext^1_{\GL_2(F), Z}[\pi_1, \chi \pi_2^\vee]$ denotes $\Ext^1$ calculated in the category of smooth
 representations of $\GL_2(F)$ with a given central character on $Z=F^\times \subset \GL_2(F)$; the above isomorphism
 is Proposition 2.4 of \cite{CF2}.

 It is easy to see that if $\pi_1,\pi_2$ are irreducible and infinite dimensional representations of $\GL_2(F)$, then 
 $\Ext^1_{\GL_2(F), Z}[\pi_1, \chi \pi_2^\vee]$ is not zero if and only if $\pi_1,\pi_2^\vee$ are irreducible principal series representations
 of $\GL_2(F)$ such that  $\pi_1 \cong \chi \pi_2^\vee$ with $\chi$ a quadratic character. (Vanishing of $\Ext^1_{\PGL_2(F)}[\St_2, \chi \St_2]$ is a well-known generality about vanishing of
 $\Ext^i_G(\pi_1,\pi_2), i \geq 1$, among discrete series representations of a semisimple group $G$;
 for a proof in this case, see Lemma 7 of \cite{Pr2}.) 
\end{proof}

\begin{cor} \label{gl2Bessel} Let $T$ be the split torus of $\PGL_2(F)$ consisting of the group of 
  diagonal matrices. Then for $\chi$ a character of $T$, the Bessel model for $\PGL_2(F)$,
  \[ \ind_T^{\PGL_2(F)} (\chi),\]
  is a projective module for $\PGL_2(F)$ if and only if $\chi^2 \not = 1$. 
  \end{cor}
\begin{proof} By Proposition \ref{gl2} applied to $\pi_1 =  \ind_B^{\PGL_2(F)}(\chi)$,  an irreducible principal
  series representation of $\PGL_2(F)$,  and $\pi_2=\St_2$, the Steinberg representation of $\PGL_2(F)$,
  $\pi_1 \otimes \pi_2$ is a projective module for $\PGL_2(F)$. Note that:
  \[ \St_2|_{B} \cong \ind_T^B(\CC),\]
  where $B$ is the Borel subgroup of $\PGL_2(F)$ containing $T \cong F^\times$,
    and
 $\ind_T^B(\CC)$ is the unnormalized induction of the trivial character of $T$.
  This follows by realising the Steinberg representation as a  quotient of the
  (un-normalised) induced
  representation  of $\PGL_2(F)$ induced by the trivial character of $B$, and applying the Mackey theory. 
  Since $\pi_1 = \ind_B^{\PGL_2(F)}(\chi)$, it follows that:
  \[ \pi_1 \otimes \St_2 = \ind_B^{\PGL_2(F)}(\chi) \otimes \St_2 =  \ind_B^{\PGL_2(F)}(\chi \otimes \St_2|_B) \cong \ind_T^{\PGL_2(F)}(\chi).\]
  As $\pi_1  = \ind_B^{\PGL_2(F)}(\chi)$ (un-normalised induction) is an irreducible principal series for
  $\chi^2 \not = 1, \nu^2$
  where $\nu: F^\times \rightarrow \CC^\times$ is the normalized absolute value on $F^\times$, Proposition \ref{gl2}  proves that  $\ind_T^{\PGL_2(F)} (\chi)$
  is a projective module for $\PGL_2(F)$ if $\chi^2 \not = 1$
  and $\chi^2 \not = \nu^2$.
  If  $\chi^2= \nu^2$, we appeal to the isomorphism:  $\ind_T^{\PGL_2(F)} (\chi) \cong \ind_T^{\PGL_2(F)} (\chi^{-1})$,
  to deduce that  $\ind_T^{\PGL_2(F)} (\chi)$
  is a projective module for $\PGL_2(F)$ if $\chi^2 \not = 1$.

  Finally, to show that  $\ind_T^{\PGL_2(F)} (\chi)$ are not projective modules for $\chi^2=1$,
  it clearly suffices to do this for $\chi$ the trivial character of $T$, in which case the earlier
  analysis \[ \ind_B^{\PGL_2(F)}(\CC) \otimes \St_2 =  \ind_B^{\PGL_2(F)}(\St_2|_B) \cong \ind_T^{\PGL_2(F)}(\CC),\] continues to hold, but now, we have the exact sequence (un-normalized induction):
  \[ 0 \rightarrow \St_2 \rightarrow  \ind_B^{\PGL_2(F)}(\CC) \otimes \St_2 \rightarrow
  \St_2 \otimes \St_2 \rightarrow 0,\]
   in which 
  by Proposition \ref{gl2},  $\St_2 \otimes \St_2$ is a projective module, 
  therefore,  \[\ind_B^{\PGL_2(F)}(\CC) \otimes \St_2 \cong  \St_2 \otimes \St_2 + \St_2,\]
which is clearly not a projective module! 
  \end{proof}
  
\begin{remark}
  Corollary \ref{gl2Bessel} deals with a case omitted in Corollary \ref{split-SO(W)} when $\SO(W)$ is the split $\SO_2(F)$. In the present case,  $\ind_T^{\PGL_2(F)} (\CC)$ is a projective module in all Bernstein components
  except the Iwahori block.
\end{remark}

\begin{remark}
  Corollary \ref{gl2Bessel} proves  a very special case of Conjecture \ref{integral}(1), in fact the simplest case of it, dealing with the pair $(\SO_3(F),\SO_2(F))$ both split, and is equivalent to the assertion that $\Ext^i_{\SO_2(F)}(\pi,\chi)=0$ for $i>0$ where $\pi$ is any
  infinite dimensional representation of $\SO_3(F)=\PGL_2(F)$.   If the assertion of Corollary \ref{gl2Bessel} looks a bit different by the presence of the condition $\chi^2 \not = 1$, it is because projectivity of a representation $\pi$ (say of a group $G$)
  requires $\Ext^i_G(\pi,\pi')=0$  for all $i >0$, and all irreducible $\pi'$,
  whereas  Conjecture \ref{integral}(1) is only about vanishing  $\Ext_G^i(\pi,\pi')$ for certain $\pi'$ (which belong to
  generic $L$-packets).
    \end{remark}

  \vspace{2mm}

  Here is an application of the calculation on Ext groups which when combined with the
  duality theorem leads to existence of submodules. 
The 
following proposition gives a complete classification
of irreducible submodules $\pi$ of the tensor product $\pi_1 \otimes \pi_2$ of two (irreducible, infinite dimensional)
representations $\pi_1,\pi_2$ of $\GL_2(F)$ with the product of their central characters trivial. A more general result is available in \cite{CF}.

\begin{prop}
Let $\pi_1, \pi_2$ be two irreducible admissible infinite dimensional 
representations of $\GL_2(F)$ with product of their
central characters trivial. Then the following is a  complete list of irreducible sub-representations
$\pi$ of $\pi_1 \otimes \pi_2$ as $\PGL_2(F)$-modules.

\begin{enumerate}
\item $\pi$ is a supercuspidal representation of $\PGL_2(F)$, and appears as a quotient of $\pi_1 \otimes \pi_2$.  

\item $\pi$ is a twist of the Steinberg representation, which we assume by absorbing the twist in 
$\pi_1$ or $\pi_2$ to be the Steinberg representation $\St$ of $\PGL_2(F)$. Then $\St$ is a 
submodule of $\pi_1 \otimes \pi_2$ if and only if 
$\pi_1,  \pi_2$ are both irreducible
principal series representations, and $\pi_1 \cong \pi_2^\vee$.
\end{enumerate}
\end{prop} 

\begin{remark}
  Unlike the case of triple products above, Chan in \cite{Chan} has proved that in the case of the pair $(\GL_{n+1}(F),\GL_n(F))$,
  if $\Hom_{\GL_n(F)}[\pi_2,\pi_1] \not = 0$, for $\pi_1$ an irreducible representation of $\GL_{n+1}(F)$ and $\pi_2$ of $\GL_n(F)$, then both $\pi_1,\pi_2$ must be one dimensional. Thus in this case, even supercuspidals of $\GL_n(F)$ do not arise as submodules which is related to the non-compact center of the subgroup $\GL_n(F)$
  (and which is not contained in the center of the ambient $\GL_{n+1}(F)$).  
  
  \end{remark}

\begin{remark}
  By Theorem \ref{CS}, for  $\pi$ a generic representation of
  $\GL_{n+1}(F)$,  if $\pi|_{\GL_n(F)}$
  is  a projective representation of $\GL_n(F)$ in any particular block
  of representations of $\GL_n(F)$, then in that block $\pi|_{\GL_n(F)}$ is independent of the generic representation
  $\pi$ of $\GL_{n+1}(F)$. However, this is not true for the restriction problem studied
  in this section, say from $\PGL_2(F) \times \PGL_2(F)$ to the diagonal $\PGL_2(F)$.
  To see this, let $\chi$ be the unique quadratic unramified character of $F^\times$. Let $\pi_1$ be a cuspidal representation
  of $\PGL_2(F)$ for which $\chi \pi_1 \cong \pi_1$, and let $\pi_2$ be a cuspidal representation
  of $\PGL_2(F)$ for which $\chi \pi_2 \not \cong \pi_2$.
  Then the representations
  $\pi_1 \boxtimes \pi_1$, $\pi_1 \boxtimes \pi_2$,   $ \pi_2 \boxtimes \pi_2$, and  $\chi \pi_2 \boxtimes \pi_2$ of  $\PGL_2(F) \times \PGL_2(F)$,
  have different branching when restricted to the subgroup  $\Delta \PGL_2(F)$ for
  the irreducible representations
$\Delta \PGL_2(F)$
  involving the trivial representation, and the character $\chi$, both
  considered as characters of $\Delta \PGL_2(F)$.
  More precisely, for $\Pi$ one of these 4 representations, $\dim\Hom_{\PGL_2(F)}(\Pi,\CC) \leq 1$, and
  $\dim\Hom_{\PGL_2(F)}(\Pi,\CC_\chi) \leq 1$, and all the 4 options on dimension of the Hom spaces
  hold,  giving 4 distinct modules in the Iwahori block of $\PGL_2(F)$ which are all projective
    modules of $\PGL_2(F)$ by Proposition \ref{gl2}.

    The paper \cite{BS} (Remark 3.2, Proposition 3.4  and Corollary 3.7) has general results involving
  projective modules 
  of classical groups which have multiplicity 1, in particular they prove that there are exactly 4 projective modules in the Iwahori block  of $\PGL_2(F)$ which have multiplicity 1. It is thus an interesting exercise to
  independently prove that when $V_1\otimes V_2$
  is a projective module of $\PGL_2(F)$, as in Proposition \ref{gl2}, then $V_1\otimes V_2$ is one of these 4 representations in the
  Iwahori block, something which this author has been able to do only when one of the representations, say $V_1$ is a principal
  series, in which case to understand $V_1\otimes V_2$ one needs to understand $V_2$ restricted to the Borel subgroup
  of $\PGL_2(F)$, which is
  part of the Kirillov theory. Thus we have a complete understanding of the structure of $V_1\otimes V_2$ where $V_1,V_2$ are
  irreducible admissible representations of $\PGL_2(F)$ not only through their irreducible quotients
  which was done in \cite{Pr1}.      \end{remark}

\begin{remark}
  It would be nice to understand projective modules $\Pi$
  of reductive $p$-adic groups $G$ which have multiplicity $\leq 1$ (i.e., $\dim \Hom_G(\Pi,\pi) \leq 1$ for all irreducible representations $\pi$ of $G$) as being the closest
  analogue of rank 1 projective modules of commutative rings. Are there finitely many in any given block? Is there at least one in any block?
  Can one classify them in some ways? For the case of  $\GL_n(F)$, by Theorem \ref{CS}(3) due Chan and Savin,
  there are exactly two  such projective modules up to isomorphism
  in the Iwahori block of $\GL_n(F)$:
  the induced representations from the trivial and the Steinberg representation of $\GL_n(\O_F)$.
  This and the $\PGL_2(F)$ example of the previous remark suggests that, up to a character twist
  and automorphisms of $G(F)$,
  these are the only two projective modules in the Iwahori block of a split reductive group with
  multiplicity $\leq 1$. The situation for general reductive groups is not clear. At least there is the
  well-known generality for simply connected $G$ that there is a compact open subgroup whose trivial representation
  appears with multiplicity at most 1 in any irreducible smooth representation of $G$, thus
  the corresponding induced representation will give a projective module of $G$ with multiplicity $\leq 1$.

  There is also the question if understanding multiplicities
  of (all) irreducible quotients (if all of them are finite) determines the module up to isomorphism, let us
  say in the GGP context, where the question seems to have an affirmative answer. (For example, Lemma \ref{commutative} determines modules over a commutative
  ring in terms of its irreducible quotients.) 
 \end{remark}
\section{Template from algebraic geometry}
        
            We enumerate some of the basic theorems in algebraic geometry which seem to have
    closely related analogues in our context.   For the analogy, we consider
    $H^0(X,{\mathfrak F})$,
    for $X$ a
    smooth projective varieties (or sometimes more general varieties) equipped
    with a coherent sheaf ${\mathfrak F}$ versus $\Hom[\pi_1,\pi_2]$, and corresponding $H^i$
    and $\Ext^i$.  
    \begin{enumerate}
    \item Finite dimensionality of $H^i(X,{\mathfrak F})$
      and  vanishing for $i> \dim X$. 
    \item Semi-continuity theorems available both in algebraic geometry for    
      $H^i(X,{\mathfrak F}_\lambda)$, and $\Ext^i[\pi_{1, \lambda}, \pi_{2,\mu}]$ for families of sheaves or
      of representations.
      
    \item Riemann-Roch theorem  expressing $\EP(X,{\mathfrak F})$ in terms of simple
      invariants associated to $X$ and the sheaf ${\mathfrak F}$. In our case, these are the integral formulae
      which go into the Kazhdan conjecture and in the work of  Waldspurger, involving invariants of the
      space $X$, certain elliptic tori, and invariants associated to sheaves= representations through character
      theory.  

    \item Kodaira vanishing for $H^i(X,{\mathfrak F})$, $i> 0$ for an ample sheaf ${\mathfrak F}$. 

    \item Serre duality 
    \[\Ext^i({\mathcal O}_X,{\mathfrak F}) \times \Ext^{d-i}({\mathfrak F}, \omega_X)
\rightarrow \Ext^d({\mathcal O}_X, \omega_x) = F.\] 

\item Special role played by $X = {\mathbb P}^d(F)$ in Algebraic geometry, and here,
  we have our own, {\it
      her all-embracing majesty}, $\GL_n(F)$. 

    \end{enumerate}

    \vspace{1cm}

 {\bf Acknowledgement:}
    This paper is an expanded and written version of my lecture
    in the IHES Summer School on the Langlands Program in July 2022. The author would like to thank the organizers for putting together a wonderful program. The author especially thanks R. Beuzart-Plessis, Kei Yuen Chan and Gordan Savin
    for all their
    helpful remarks. Although Dragos Fratila was not directly involved with this project, I have benefited greatly by discussions with him on our joint work \cite{FP}, and some follow up work, which is also on homological aspects of $p$-adic groups
    (though not about branching laws or the GGP),
    and it is possible that some of our discussions have found a place here; I thank him warmly! The paper has benefited by the comments of the referee whom the author thanks profusely.

\bibliographystyle{amsalpha}

\end{document}